\documentclass{article} 
\usepackage{nips15submit_e,times}

\title{Linear Convergence of Variance-Reduced Stochastic Gradient without Strong Convexity}

\author{
Pinghua Gong,~~~Jieping Ye\\
Department of Computational Medicine and Bioinformatics\\
University of Michigan, MI 48109\\
gongp@umich.edu,~~jpye@umich.edu
}

\usepackage[ruled,linesnumbered,titlenumbered]{algorithm2e}
\usepackage{graphicx}
\usepackage{amsmath,amssymb,url}
\usepackage{float}
\usepackage{bm}
\usepackage{multirow}
\usepackage{wrapfig}

\def\Pr{\mathrm{Pr}}

\newcommand{\Expectg}[1]{\mathbb{E}\left[#1\right]}
\newcommand{\Expect}[2]{\mathbb{E}_{#1}\left[#2\right]}

\newcommand{\Argmin}{\mathop{\arg\min}}

\newcommand{\EqRef}[1]{Eq.~(\ref{#1})}
\newcommand{\SecRef}[1]{Section~\ref{#1}}
\newcommand{\ApdRef}[1]{Supplement~\ref{#1}}
\newcommand{\ThemRef}[1]{Theorem~\ref{#1}}
\newcommand{\LemmaRef}[1]{Lemma~\ref{#1}}
\newcommand{\DefRef}[1]{Definition~\ref{#1}}

\newcommand{\RemarkRef}[1]{Remark~\ref{#1}}
\newcommand{\FigRef}[1]{Figure~\ref{#1}}
\newcommand{\AlgRef}[1]{Algorithm~\ref{#1}}

\newtheorem{theorem}{Theorem}
\newtheorem{lemma}{Lemma}

\newtheorem{definition}{Definition}

\newtheorem{remark}{Remark}

\newenvironment{proof}[1][Proof]{\begin{trivlist}
\item[\hskip \labelsep {\bfseries #1}]}{\end{trivlist}}

\def\qed {{
\parfillskip=0pt 
\widowpenalty=10000 
\displaywidowpenalty=10000 
\finalhyphendemerits=0 
\leavevmode 
\unskip 
\nobreak 
\hfil 
\penalty50 
\hskip.2em 
\null 
\hfill 
$\square$
\par}} 

\nipsfinalcopy 

\begin{document}

\maketitle

\begin{abstract}
Stochastic gradient algorithms estimate the gradient based on only one or a few samples
and enjoy low computational cost per iteration. They have been widely used in large-scale optimization problems.
However, stochastic gradient algorithms are usually slow to converge and achieve sub-linear convergence rates,
due to the inherent variance in the gradient computation. To accelerate the convergence,
some variance-reduced stochastic gradient algorithms, e.g., proximal stochastic variance-reduced gradient (Prox-SVRG) algorithm,
have recently been proposed to solve strongly convex problems. Under the strongly convex condition,
these variance-reduced stochastic gradient algorithms achieve a linear convergence rate. However,
many machine learning problems are convex but not strongly convex. In this paper,
we introduce Prox-SVRG and its projected variant called Variance-Reduced Projected Stochastic Gradient (VRPSG) to
solve a class of non-strongly convex optimization problems widely used in machine learning. As the main technical contribution of this paper,
we show that both VRPSG and Prox-SVRG achieve a linear convergence rate without strong convexity.
A key ingredient in our proof is a Semi-Strongly Convex (SSC) inequality which is the first to be rigorously proved for a class of non-strongly convex problems in both constrained
and regularized settings. Moreover, the SSC inequality is independent of algorithms and may be applied to analyze other stochastic gradient algorithms
besides VRPSG and Prox-SVRG, which may be of independent interest.
To the best of our knowledge, this is the first work that establishes the linear convergence rate for the
variance-reduced stochastic gradient algorithms on solving both constrained and regularized problems without strong convexity.
\end{abstract}

\section{Introduction}\label{sec:introduction}
Convex optimization has played an important role in machine learning as many machine learning problems can be cast into a convex optimization problem.
Nowadays the emergence of big data makes the optimization problem challenging to solve and first-order stochastic gradient algorithms are often
preferred due to their simplicity and low per-iteration cost. The stochastic gradient algorithms estimate the gradient based on only one or a few samples,
and have been extensively studied in large-scale optimization problems \cite{zhang2004solving,duchi2009efficient,hu2009accelerated,xiao2010dual,hazan2011beyond,rakhlin2011making,lan2012optimal,shamir2013stochastic}.
In general, the standard stochastic gradient algorithm randomly draws only one or a few samples at each iteration to compute the gradient and then update
the model parameter. The standard stochastic gradient algorithm estimates the gradient without involving all samples and the computational cost per iteration is independent
of the sample size. Thus, it is very suitable for large-scale problems. However, the standard stochastic gradient algorithms usually suffer from slow convergence.
In particular, even under the strongly convex condition, the convergence rate of standard stochastic gradient algorithms is only sub-linear.
In contrast, it is well-known that full gradient descent algorithms can achieve linear convergence rates
with the strongly convex condition \cite{nesterov2013gradient}. It has been recognized that the slow convergence of
the standard stochastic gradient algorithm results from the inherent variance in the gradient evaluation.
To this end, some (implicit or explicit) variance-reduced stochastic gradient algorithms have been proposed recently; examples include
Stochastic Average Gradient (SAG) \cite{le2012stochastic}, Stochastic Dual Coordinate Ascent (SDCA) \cite{shalev2012proximal,shalev2013stochastic},
Epoch Mixed Gradient Descent (EMGD) \cite{zhang2013linear}, Stochastic Variance Reduced Gradient (SVRG) \cite{johnson2013accelerating},
Semi-Stochastic Gradient Descent (S2GD) \cite{konevcny2013semi} and Proximal Stochastic Variance Reduced Gradient (Prox-SVRG) \cite{xiao2014proximal}.
Under the strongly convex condition, these variance-reduced stochastic gradient algorithms achieve linear
convergence rates. However, in practical problems, many objective functions to be minimized are convex but not strongly convex.
For example, the least squares regression and logistic regression problems are extensively studied and
both objective functions are not strongly convex
when the data matrix is not full column rank.
Moreover, even without the strongly convex condition, linear convergence rates can be achieved for some full
gradient descent algorithms \cite{luo1992convergence,luo1993error,tseng2009coordinate,tseng2010approximation,hou2013linear,wang2014iteration}.
This motivates us to address the following question: can some variance-reduced stochastic gradient algorithms achieve a linear convergence rate
under mild conditions but without strong convexity?

In this paper, we adopt Prox-SVRG \cite{xiao2014proximal} and its projected variant called Variance-Reduced Projected Stochastic Gradient (VRPSG)
to solve a class of non-strongly convex optimization problems.
Our major technical contribution is to establish a linear convergence rate for both VRPSG and Prox-SVRG without strong convexity.
The key challenge to prove the linear convergence for both VRPSG and Prox-SVRG lies in how to establish a Semi-Strongly Convex (SSC) inequality
which provides an upper bound of the distance of any feasible solution
to the optimal solution set by the gap between the objective function value at that feasible solution and the optimal objective function value.
The SSC inequality can be easily established under the condition that the objective function is strongly convex.
However, it is not the case without the strongly convex condition.
To the best of our knowledge, we are the first
to rigorously prove the SSC inequality for a class of non-strongly convex problems in both constrained and regularized settings.
Moreover, the SSC inequality may be applied to analyze other stochastic gradient algorithms besides VRPSG and Prox-SVRG,
which may be of independent interest (see \RemarkRef{remark:sscineq} in \SecRef{sec:sscineq}).
Note that existing convergence analyses for
full gradient methods \cite{luo1992convergence,luo1993error,tseng2009coordinate,tseng2010approximation,hou2013linear,wang2014iteration}
cannot be directly extended to the stochastic setting as they rely on a different inequality involving full gradient.
Thus, it is nontrivial to establish the linear convergence rate for the variance-reduced stochastic gradient
algorithms on solving both constrained and regularized problems without strong convexity.

\section{Linear Convergence of VRPSG and Prox-SVRG}\label{sec:vrpsgalg}
We first present both constrained and regularized optimization problems\footnote{We present constrained and regularized problems separately, however the analysis of regularized problems depends on some key results established for constrained problems.}, discuss some mild assumptions about the problems and show some examples that
satisfy the assumptions. Then we present the Semi-Strongly Convex (SSC) inequality and introduce VRPSG and Prox-SVRG to solve
the optimization problems. Finally, we state our main results on linear convergence for both VRPSG and Prox-SVRG algorithms. The proofs are deferred to the following section.

\subsection{Optimization Problems, Assumptions and Examples}\label{sec:vrpsg.assumption}
\textbf{Constrained Problems}: We first consider the following constrained optimization problem:
\begin{align}
\min_{\mathbf{w}\in\mathcal{W}}\left\{f(\mathbf{w})=h(X\mathbf{w})+\mathbf{q}^T\mathbf{w}\right\},
\mathrm{where}~\mathbf{w,q}\in\mathbb{R}^{d},~X\in\mathbb{R}^{n\times d},\label{eq:constrainedprob}
\end{align}
and make the following assumptions on the above problem:
\begin{itemize}
\item [\textbf{A1}]$f(\mathbf{w})$ is the average of $n$ convex components $f_i(\mathbf{w})$, that is,
$f(\mathbf{w})=\frac{1}{n}\sum_{i=1}^nf_i(\mathbf{w})$,
where $\nabla f(\mathbf{w})$ and $\nabla f_i(\mathbf{w})$ are Lipschitz continuous with constants $L$ and $L_i$, respectively.
\item[\textbf{A2}] The effective domain $\mathrm{dom}(h)$ of $h$ is open and non-empty. Moreover,
$h(\mathbf{u})$ is continuously differentiable on $\mathrm{dom}(h)$ and strongly convex on any convex compact subset of $\mathrm{dom}(h)$.
\item[\textbf{A3}] The constraint set $\mathcal{W}$ is a polyhedral set, e.g. $W=\left\{\mathbf{w}\in\mathbb{R}^d: C\mathbf{w}\leq\mathbf{b}\right\}$ for some $C\in\mathbb{R}^{l\times d},\mathbf{b}\in\mathbb{R}^{l}$.
Moreover, the optimal solution set $\mathcal{W}^\star_c$ to \EqRef{eq:constrainedprob} is non-empty.
\end{itemize}
\vspace{-0.0cm}
\begin{remark}
Assumption \textbf{A2} indicates that $h(\mathbf{u})$ may not be strongly convex on $\mathrm{dom}(h)$ but strictly convex on $\mathrm{dom}(h)$.
Notice that $f(\cdot)$ is convex, so $\mathcal{W}^\star_c$ must be convex and the Euclidean projection of any $\mathbf{w}\in\mathbb{R}^d$ onto $\mathcal{W}^\star_c$ is unique.
Moreover, for any finite $\mathbf{w},\mathbf{u}\in\mathcal{W}$, $X\mathbf{w}$ and $X\mathbf{u}$ must belong to a convex compact subset $\mathcal{U}\subseteq\mathrm{dom}(h)$.
Thus, by assumption \textbf{A2}, there exists a $\mu>0$ such that
\begin{align}
h(X\mathbf{w})\geq & h(X\mathbf{u})+\nabla h(X\mathbf{u})^T(X\mathbf{w}-X\mathbf{u})+\frac{\mu}{2}\|X\mathbf{w}-X\mathbf{u}\|^2,~\forall X\mathbf{w},X\mathbf{u}\in\mathcal{U}.\nonumber
\end{align}
\end{remark}

\textbf{Example 1 (Constrained Problems)}: There are many examples that satisfy assumptions \textbf{A1-A3}, including three popular problems: $\ell_1$-constrained least squares (i.e., Lasso \cite{tibshirani1996regression}),
$\ell_1$-constrained logistic regression and the dual problem of linear SVM. Specifically, for the
$\ell_1$-constrained least squares: the objective function is $f(\mathbf{w})=\frac{1}{2n}\|X\mathbf{w}-\mathbf{y}\|^2$; the convex component is $f_i(\mathbf{w})=\frac{1}{2}(\mathbf{x}_i^T\mathbf{w}-y_i)^2$,
where $\mathbf{x}_i^T$ is the $i$-th row of $X$; the strongly convex function is $h(\mathbf{u})=\frac{1}{2n}\|\mathbf{u}-\mathbf{y}\|^2$; the polyhedral set is
$\mathcal{W}=\{\mathbf{w}:\|\mathbf{w}\|_1\leq\tau\}=\{\mathbf{w}:C\mathbf{w}\leq\mathbf{b}\}$, where each row of $C\in\mathbb{R}^{2^d\times d}$ is a $d$-tuples of the form $[\pm1,\cdots,\pm1]$,
and each entry of $\mathbf{b}\in\mathbb{R}^{2^d}$ is $\tau$.
For the $\ell_1$-constrained logistic regression: the objective function is $f(\mathbf{w})=\frac{1}{n}\sum_{i=1}^n\log(1+\exp(-y_i\mathbf{x}_i^T\mathbf{w}))$; the convex component is $f_i(\mathbf{w})=\log(1+\exp(-y_i\mathbf{x}_i^T\mathbf{w}))$,
where $X=[\mathbf{x}_1^T;\cdots;\mathbf{x}_n^T]^T$; the strongly convex function\footnote{The function $h(\mathbf{u})=\frac{1}{n}\sum_{i=1}^n\log(1+\exp(-y_iu_i))$ is strictly convex on
$\mathbb{R}^n$ and strongly convex on any convex compact subset of $\mathbb{R}^n$.} is $h(\mathbf{u})=\frac{1}{n}\sum_{i=1}^n\log(1+\exp(-y_iu_i))$; the polyhedral set is the same as the $\ell_1$-constrained least squares.
For the dual problem of linear SVM, the objective is a convex quadratic function which satisfies assumptions \textbf{A1-A2}; the constraint set is $\mathcal{W}=\{\mathbf{w}:l_i\leq w_i\leq u_i\}$ with $l_i\leq u_i~(i=1,\cdots,d)$, which satisfies the assumption \textbf{A3}.
Additional constraint sets that satisfy the assumption \textbf{A3} also include $\ell_{1,\infty}$-ball set $\mathcal{W}=\{\mathbf{w}:\sum_{i=1}^{T}\|\mathbf{w}_{\mathcal{G}_i}\|_{\infty}\leq\tau\}$ with $\cup_{i=1}^T\mathcal{G}_i=\{1,\cdots,d\}$ and $\mathcal{G}_i\cap\mathcal{G}_j=\emptyset$ for $i\neq j$ \cite{quattoni2009efficient}.

\textbf{Regularized Problems}: Now let us consider the following regularized optimization problem:
\begin{align}
\min_{\mathbf{w}\in\mathbb{R}^d}\left\{F(\mathbf{w})=f(\mathbf{w})+r(\mathbf{w})=h(X\mathbf{w})+r(\mathbf{w})\right\},~\text{where}~X\in\mathbb{R}^{n\times d}, \label{eq:regularizedprob}
\end{align}
and we make the following additional assumption besides assumptions \textbf{A1,~A2}:
\begin{itemize}
\item[\textbf{B1}] $r(\mathbf{w})$ is convex; the epigraph of $r(\mathbf{w})$ defined by $\{(\mathbf{w},\varpi): r(\mathbf{w})\leq \varpi\}$ is a polyhedral set and the optimal solution set $\mathcal{W}^\star_r$ to \EqRef{eq:regularizedprob} is non-empty.
\end{itemize}
\textbf{Example 2 (Regularized Problems)}: Examples that satisfy assumptions \textbf{A1,~A2} and \textbf{B1} include $\ell_1$ ($\ell_{1,\infty}$)-regularized least squares and logistic regression problems.

\subsection{Semi-Strongly Convex (SSC) Problem and Inequality}
Let us now introduce the Semi-Strongly Convex (SSC) property.
\begin{definition}\label{def:sscineq}
The problem in \EqRef{eq:constrainedprob} is SSC if there exists a constant $\beta>0$ such that for any finite $\mathbf{w}\in\mathcal{W}$:
\begin{align}
f(\mathbf{w})-f^\star\geq\frac{\beta}{2}\left\|\mathbf{w}-\Pi_{\mathcal{W}^\star_c}(\mathbf{w})\right\|^2,~\text{where}~f^\star~\text{is the optimal objective function value of \EqRef{eq:constrainedprob}}.\nonumber
\end{align}
The problem in \EqRef{eq:regularizedprob} is SSC if for any finite $\mathbf{w}\in\mathbb{R}^d$, there exists a constant $\beta>0$ such that
\begin{align}
F(\mathbf{w})-F^\star\geq\frac{\beta}{2}\left\|\mathbf{w}-\Pi_{\mathcal{W}^\star_r}(\mathbf{w})\right\|^2,~\text{where}~F^\star~\text{is the optimal objective function value of \EqRef{eq:regularizedprob}}.\nonumber
\end{align}
\end{definition}
In \SecRef{sec:sscineq}, we will rigorously prove that the problem in \EqRef{eq:constrainedprob} is SSC under assumptions \textbf{A1-A3} and the problem in \EqRef{eq:regularizedprob} is SSC under assumptions \textbf{A1, A2} and \textbf{B1}, which is a key to show the linear convergence of VRPSG and Prox-SVRG to be given below. To the best of our knowledge, we are the first to provide a rigorous proof of the
SSC inequality for both constrained and regularized problems in \EqRef{eq:constrainedprob} and \EqRef{eq:regularizedprob} without strong convexity.
Moreover, the SSC inequality may be of independent interest, as it may be applied to analyze other stochastic gradient algorithms (see \RemarkRef{remark:sscineq} in \SecRef{sec:sscineq}).
\subsection{Algorithms and Main Results}
\textbf{VRPSG for solving \EqRef{eq:constrainedprob}}: A standard stochastic method for solving \EqRef{eq:constrainedprob} is the projected stochastic gradient algorithm which generates the sequence $\{\mathbf{w}^k\}$ as follows:
\begin{align}
\mathbf{w}^{k}=\Pi_{\mathcal{W}}(\mathbf{w}^{k-1}-\eta_k\nabla f_{i_k}(\mathbf{w}^{k-1})),\label{eq:projstocgrad}
\end{align}
where $i_k$ is randomly drawn from $\{1,\cdots,n\}$ in uniform. At each iteration, the projected stochastic gradient algorithm
computes the gradient involving only a single sample and thus is suitable for large-scale problems with large $n$. Although we have
an unbiased gradient estimate at each step, i.e., $\Expectg{\nabla f_{i_k}(\mathbf{w}^{k-1})}=\nabla f(\mathbf{w}^{k-1})$, the variance $\Expectg{\|\nabla f_{i_k}(\mathbf{w}^{k-1})-\nabla f(\mathbf{w}^{k-1})\|^2}$ introduced
by sampling makes the step size $\eta_k$ diminishing to guarantee convergence, which finally results in slow convergence. Therefore, the key for improving the convergence rate
of the projected stochastic gradient algorithm is to reduce the variance by sampling.
Motivated by the variance-reduce techniques in \cite{johnson2013accelerating,xiao2014proximal},
we consider a projected variant of Prox-SVRG \cite{xiao2014proximal} called Variance-Reduced Projected Stochastic Gradient (VRPSG) (in \AlgRef{alg:projstocgrad}) to efficiently solve \EqRef{eq:constrainedprob} [i.e., VRPSG is equivalent to Prox-SVRG by using a proximal step instead of the projection step in \AlgRef{alg:projstocgrad} (Line 10)].
Both VRPSG and Prox-SVRG employ a two-layer loop to reduce the variance. We have the following convergence result:

\vspace{-0.1cm}
\begin{algorithm}[H]\label{alg:projstocgrad}
Choose the update frequency $m$ and the learning rate $\eta$\;
Initialize $\tilde{\mathbf{w}}^0\in\mathcal{W}$\;
Choose $p_i\in(0,1)$ such that $\sum_{i=1}^np_i=1$\;
   \For{$k=1,2,\cdots$}
   {
    $\tilde{\bm{\xi}}^{k-1}=\nabla f(\tilde{\mathbf{w}}^{k-1})$\;
    $\mathbf{w}^k_0=\tilde{\mathbf{w}}^{k-1}$\;
        \For{$t=1,2,\cdots,m$}
        {Randomly pick $i^k_t\in\{1,\cdots,n\}$ according to $P=\{p_1,,\cdots,p_n\}$\;
        $\mathbf{v}_{t}^k=(\nabla f_{i^k_t}(\mathbf{w}_{t-1}^{k}) - \nabla f_{i^k_t}(\tilde{\mathbf{w}}^{k-1}))/(np_{i^k_t}) + \tilde{\bm{\xi}}^{k-1}$\;
        $\mathbf{w}_{t}^{k}=\Pi_{\mathcal{W}}(\mathbf{w}_{t-1}^{k}-\eta\mathbf{v}_{t}^k)=\Argmin_{\mathbf{w}\in\mathcal{W}}\frac{1}{2}\|\mathbf{w}-(\mathbf{w}_{t-1}^{k}-\eta\mathbf{v}_{t}^k)\|^2$\;
        }
    $\tilde{\mathbf{w}}^{k}=\frac{1}{m}\sum_{t=1}^{m}\mathbf{w}_{t}^{k}$\;
   }
\caption{VRPSG: Variance-Reduced Projected Stochastic Gradient}
\end{algorithm}\vspace{-0.1cm}

\begin{theorem}\label{theorem:linearconvergence}
Let $\mathbf{w}^\star\in\mathcal{W}^\star_c$ be any optimal solution to \EqRef{eq:constrainedprob}, $f^\star=f(\mathbf{w}^\star)$ be the optimal objective function value in \EqRef{eq:constrainedprob}
and $L_P=\max_{i\in\{1,\cdots,n\}}[L_i/(np_i)]$ with $p_i\in(0,1),\sum_{i=1}^np_i=1$. In addition, let $0<\eta<1/(4L_P)$ and
\begin{align}
\beta=\frac{1}{\theta^2\left(\frac{1+2\|\nabla h(\mathbf{r}^\star)\|^2}{\mu}+M\right)},\label{eq:betadef}
\end{align}
where $\theta>0$ is a constant whose estimate is provided in \LemmaRef{lemma:hoffmanbound} and \RemarkRef{remark:hoffmanbound} in \ApdRef{appendix:auxlemmas};
$\mu>0$ is the strongly convex modulus of $h(\cdot)$ in some convex compact set; $M>0$ is an upper bound of $f(\mathbf{w})-f^\star$ for any $\mathbf{w}\in\mathcal{W}$;
$\mathbf{r}^\star$ is a constant vector such that $X\mathbf{w}^\star=\mathbf{r}^\star$ for all $\mathbf{w}^\star\in\mathcal{W}^\star_c$ (refer to \LemmaRef{lemma:globalsolution} in
\ApdRef{appendix:auxlemmas} for more details about $\mathbf{r}^\star$).
If $m$ is sufficiently large such that
\begin{align}
\rho=\frac{4L_P\eta(m+1)}{(1-4L_P\eta)m}+\frac{1}{\beta\eta(1-4L_P\eta)m}<1,\label{eq:linearrate}
\end{align}
then under the assumption that $\{\mathbf{w}^k_t\}$ is bounded and $\textbf{A1}-\textbf{A3}$ hold,
the VRPSG algorithm (summarized in \AlgRef{alg:projstocgrad}) achieves a linear convergence rate in expectation:
\begin{align}
\Expect{\mathcal{F}^k_m}{f(\tilde{\mathbf{w}}^k)-f^\star}\leq\rho^k(f(\tilde{\mathbf{w}}^0)-f^\star),\nonumber
\end{align}
where $\tilde{\mathbf{w}}^k$ is defined in \AlgRef{alg:projstocgrad} and $\Expect{\mathcal{F}^k_m}{\cdot}$ denotes the expectation with respect to
the random variable $\mathcal{F}^k_m$ with $\mathcal{F}^k_t~(1\leq t\leq m)$ being defined as
$\mathcal{F}^k_t=\{i^1_1,\cdots,i^1_m,i^2_1,\cdots,i^{2}_m,\cdots,i^{k-1}_1,\cdots,i^{k-1}_m,i^k_1,\cdots,i^{k}_t\}$
and $\mathcal{F}^k_0=\mathcal{F}^{k-1}_m$,
where $i^k_t$ is the sampling random variable in \AlgRef{alg:projstocgrad}.
\end{theorem}
Note that the linear convergence rate $\rho$ in \EqRef{eq:linearrate} is the same with that of Prox-SVRG in \cite{xiao2014proximal}, except that
the constant $\beta>0$ is slightly more complicated. This is expected since
our convergence analysis does not require the strongly convex condition. Interested readers may refer to \ApdRef{appendix:remarktheorem} and \cite{xiao2014proximal} for more details
about the above linear convergence.

\textbf{Prox-SVRG for solving \EqRef{eq:regularizedprob}}: We use Prox-SVRG to solve \EqRef{eq:regularizedprob} [i.e., using \AlgRef{alg:projstocgrad} to solve the regularized problem in \EqRef{eq:regularizedprob} by replacing the projection step in \AlgRef{alg:projstocgrad} (Line 10) with the following proximal step]:
\begin{align}
\mathbf{w}^k_t=\Argmin_{\mathbf{w}}\left\{\frac{1}{2\eta}\left\|\mathbf{w}-(\mathbf{w}_{t-1}^{k}-\eta\mathbf{v}_{t}^k)\right\|^2+ r(\mathbf{w})\right\}.\label{eq:proxstep}
\end{align}
Next we show that the convergence analysis in \ThemRef{theorem:linearconvergence} can be accordingly extended to the regularized setting; the main result is summarized in the
following theorem:
\begin{theorem}\label{theorem:reglinearconvergence}
Let assumptions \textbf{A1,~A2} and \textbf{B1} hold. If we adopt Prox-SVRG to solve the regularized problem in \EqRef{eq:regularizedprob} [i.e., using \AlgRef{alg:projstocgrad} by replacing the projection step in
\AlgRef{alg:projstocgrad} (Line 10) with the proximal step in \EqRef{eq:proxstep}] and assume that $\{\mathbf{w}^k_t\}$ is bounded, then \ThemRef{theorem:linearconvergence} still holds by replacing \EqRef{eq:constrainedprob} and $f(\cdot)$ with
\EqRef{eq:regularizedprob} and $F(\cdot)$, respectively.
\end{theorem}

\section{Technical Proof}\label{sec:techproof}
The key to prove the linear convergence results is to establish the Semi-Strongly Convex (SSC) inequality in \DefRef{def:sscineq}. Note that the SSC inequality does not involve full gradient and is suitable to
prove linear convergence of stochastic gradient algorithms.
We want to emphasize that the linear convergence analysis
for full gradient methods \cite{luo1992convergence,luo1993error,tseng2009coordinate,tseng2010approximation,hou2013linear,wang2014iteration} rely on a different inequality $\|\mathbf{w}-\Pi_{\mathcal{W}^\star_c}(\mathbf{w})\|\leq\kappa\|\mathbf{w}-\Pi_{\mathcal{W}^\star_c}(\mathbf{w}-\nabla f(\mathbf{w}))\|$
involving full gradient and cannot be directly applied here. It is well-known that the SSC inequality holds under the strongly convex condition. However,
without the strongly convex condition, it is non-trivial to obtain this inequality.
We also note that the SSC inequality holds deterministically for all examples listed in \SecRef{sec:vrpsg.assumption}, thus it is significantly different
from the restricted strong convexity (RSC) in \cite{agarwal2012fast}, where RSC holds with high probability when the design matrix is sampled from a certain distribution.

\subsection{Proof of the SSC Inequality in Constrained and Regularized Settings}\label{sec:sscineq}
We first prove the SSC inequality (in \LemmaRef{lemma:strongineq}) for the problem in \EqRef{eq:constrainedprob} under assumptions \textbf{A1}-\textbf{A3}. Then
based on the key results in the proof of \LemmaRef{lemma:strongineq}, we prove the SSC inequality (in \LemmaRef{lemma:regstrongineq}) for the problem in \EqRef{eq:regularizedprob} under assumptions \textbf{A1,~A2} and \textbf{B1}, which is
a non-trivial extension (see \RemarkRef{remark:strongineq} for more details).
\begin{lemma}\label{lemma:strongineq}
(SSC inequality for constrained problems) Under assumptions \textbf{A1}-\textbf{A3}, the problem in \EqRef{eq:constrainedprob} satisfies the SSC inequality with $\beta>0$ defined in \EqRef{eq:betadef}.
\end{lemma}
\begin{proof}
Let $\bar{\mathbf{w}}=\Pi_{\mathcal{W}^\star_c}(\mathbf{w})$. If $\mathbf{w}\in\mathcal{W}^\star_c$, then $\bar{\mathbf{w}}=\mathbf{w}$ and the inequality holds for any $\beta>0$.
We next prove the inequality for $\mathbf{w}\in\mathcal{W},\mathbf{w}\notin\mathcal{W}^\star_c$. According to \LemmaRef{lemma:globalsolution} (in \ApdRef{appendix:auxlemmas}), we know that
there exist unique $\mathbf{r}^\star$ and $s^\star$ such that
$\mathcal{W}^\star_c=\{\mathbf{w}^\star:C\mathbf{w}^\star\leq\mathbf{b},X\mathbf{w}^\star=\mathbf{r}^\star,\mathbf{q}^T\mathbf{w}^\star=s^\star\}$ which is non-empty. For any $\mathbf{w}\in\mathcal{W}=\{\mathbf{w}:C\mathbf{w}\leq\mathbf{b}\}$, the Euclidean projection of $C\mathbf{w}-\mathbf{b}$ onto
the non-negative orthant, denoted by $[C\mathbf{w}-\mathbf{b}]^+$, is $\mathbf{0}$.
Considering \LemmaRef{lemma:hoffmanbound} (in \ApdRef{appendix:auxlemmas}), for $\mathbf{w}\in\mathcal{W}=\{\mathbf{w}:C\mathbf{w}\leq\mathbf{b}\}$, there exist a $\mathbf{w}^\star\in\mathcal{W}^\star_c$ and a constant $\theta>0$ such that
\begin{align}
\|\mathbf{w}-\bar{\mathbf{w}}\|^2\leq\|\mathbf{w}-\mathbf{w}^\star\|^2\leq\theta^2(\|X\mathbf{w}-\mathbf{r}^\star\|^2+(\mathbf{q}^T\mathbf{w}-s^\star)^2),\label{eq:wwbarineq}
\end{align}
where the first inequality is due to $\bar{\mathbf{w}}=\Pi_{\mathcal{W}^\star_c}(\mathbf{w})$ and $\mathbf{w}^\star\in\mathcal{W}^\star_c$.
By assumption $\textbf{A3}$, we know that $\mathcal{W}$ is compact. Thus, for any finite $\mathbf{w}\in\mathcal{W}$, both
$X\mathbf{w}$ and $X\bar{\mathbf{w}}$ belong to some convex compact subset $\mathcal{U}\subseteq\mathbb{R}^n$. Moreover,
we have $X\bar{\mathbf{w}}=\mathbf{r}^\star$.
Thus, by the strong convexity of $h(\cdot)$ on the subset $\mathcal{U}$, there exists a constant $\mu>0$ such that
\begin{align}
h(X\mathbf{w})-h(X\bar{\mathbf{w}})\geq\nabla h(X\bar{\mathbf{w}})^T(X\mathbf{w}-X\bar{\mathbf{w}})+\frac{\mu}{2}\|X\mathbf{w}-\mathbf{r}^\star\|^2,\nonumber
\end{align}
which together with $f(\mathbf{w})=h(X\mathbf{w})+\mathbf{q}^T\mathbf{w}$ implies that
\begin{align}
f(\mathbf{w})-f(\bar{\mathbf{w}})\geq\nabla f(\bar{\mathbf{w}})^T(\mathbf{w}-\bar{\mathbf{w}})+\frac{\mu}{2}\|X\mathbf{w}-\mathbf{r}^\star\|^2.\label{eq:strongconvexityineq}
\end{align}
Noticing that $\mathbf{w}\in\mathcal{W}$ and $\bar{\mathbf{w}}\in\mathcal{W}^\star_c$, we have
\begin{align}
\nabla f(\bar{\mathbf{w}})^T(\mathbf{w}-\bar{\mathbf{w}})\geq0,\nonumber
\end{align}
which together with \EqRef{eq:strongconvexityineq} implies that
\begin{align}
\frac{2}{\mu}(f(\mathbf{w})-f(\bar{\mathbf{w}}))\geq\|X\mathbf{w}-\mathbf{r}^\star\|^2.\label{eq:objgapxwr}
\end{align}
Next we establish the relationship between $(\mathbf{q}^T\mathbf{w}-s^\star)^2$ and $f(\mathbf{w})-f^\star$. We know that $\mathbf{q}=\nabla f(\bar{\mathbf{w}})-X^T\nabla h(\mathbf{r}^\star)$
and $s^\star=\mathbf{q}^T\bar{\mathbf{w}}$ by \LemmaRef{lemma:globalsolution} (in \ApdRef{appendix:auxlemmas}),
we know that
\begin{align}
\mathbf{q}^T\mathbf{w}-s^\star=\mathbf{q}^T(\mathbf{w}-\bar{\mathbf{w}})&=(\nabla f(\bar{\mathbf{w}})-X^T\nabla h(\mathbf{r}^\star))^T(\mathbf{w}-\bar{\mathbf{w}})\nonumber\\
&=\nabla f(\bar{\mathbf{w}})^T(\mathbf{w}-\bar{\mathbf{w}})-\nabla h(\mathbf{r}^\star)^T(X\mathbf{w}-\mathbf{r}^\star),\nonumber
\end{align}
which implies that
\begin{align}
(\mathbf{q}^T\mathbf{w}-s^\star)^2&=(\nabla f(\bar{\mathbf{w}})^T(\mathbf{w}-\bar{\mathbf{w}})-\nabla h(\mathbf{r}^\star)^T(X\mathbf{w}-\mathbf{r}^\star))^2\nonumber\\
&\leq 2(\nabla f(\bar{\mathbf{w}})^T(\mathbf{w}-\bar{\mathbf{w}}))^2 + 2(\nabla h(\mathbf{r}^\star)^T(X\mathbf{w}-\mathbf{r}^\star))^2,\nonumber
\end{align}
which together with
\begin{align}
0\leq\nabla f(\bar{\mathbf{w}})^T(\mathbf{w}-\bar{\mathbf{w}})\leq f(\mathbf{w})-f^\star\nonumber
\end{align}
implies that
\begin{align}
(\mathbf{q}^T\mathbf{w}-s^\star)^2\leq 2(f(\mathbf{w})-f^\star)^2 + 2\|\nabla h(\mathbf{r}^\star)\|^2\|X\mathbf{w}-\mathbf{r}^\star\|^2.\label{eq:qsineq}
\end{align}
Substituting Eqs. (\ref{eq:objgapxwr}), (\ref{eq:qsineq}) into \EqRef{eq:wwbarineq}, we have
\begin{align}
\|\mathbf{w}-\bar{\mathbf{w}}\|^2\leq2\theta^2\left(\frac{1+2\|\nabla h(\mathbf{r}^\star)\|^2}{\mu}(f(\mathbf{w})-f^\star)+(f(\mathbf{w})-f^\star)^2\right),\nonumber
\end{align}
which together with $f(\mathbf{w})-f^\star\leq M$ (\LemmaRef{lemma:boundedobjgap} in \ApdRef{appendix:auxlemmas}) implies that
\begin{align}
\|\mathbf{w}-\bar{\mathbf{w}}\|^2\leq2\theta^2\left(\frac{1+2\|\nabla h(\mathbf{r}^\star)\|^2}{\mu}+M\right)(f(\mathbf{w})-f^\star).\nonumber
\end{align}
This completes the proof of the lemma by considering the definition of $\beta$ in \EqRef{eq:betadef}.
\qed
\end{proof}

\begin{remark}\label{remark:strongineq}
Due to the projection step in \AlgRef{alg:projstocgrad} (Line 10), each iterate belongs to the constraint set $\mathcal{W}$. Moreover, the optimal solution set $\mathcal{W}^\star_c$
is an intersection of the polyhedral set $\mathcal{W}$ and an affine space $\{\mathbf{w}^\star:X\mathbf{w}^\star=\mathbf{r}^\star,\mathbf{q}^T\mathbf{w}^\star=s^\star\}$. The above fact
is critical to prove the SSC inequality for the constrained problem in \EqRef{eq:constrainedprob}. However, for the regularized problem in \EqRef{eq:regularizedprob}, no such property holds.
Thus it is much more challenging to extend the SSC inequality to the regularized problem. Interestingly, we find that
that the problem in \EqRef{eq:regularizedprob} is equivalent to the following constrained problem (the proof is provided in \LemmaRef{lemma:equivalent} in \ApdRef{appendix:auxlemmas}):
\begin{align}
\min_{(\mathbf{w},\varpi)\in\widetilde{\mathcal{W}}}\left\{\widetilde{F}(\mathbf{w},\varpi)=f(\mathbf{w})+\varpi=h(X\mathbf{w})+\varpi\right\},
\text{where}~\widetilde{\mathcal{W}}=\{(\mathbf{w},\varpi): r(\mathbf{w})\leq \varpi\}.\label{eq:equivprob}
\end{align}
Based on the above equivalence and some key results in the proof of \LemmaRef{lemma:strongineq}, we establish an SSC inequality for the regularized problem. Note that we still solve the regularized
problem in \EqRef{eq:regularizedprob} using Prox-SVRG and \EqRef{eq:equivprob} is only used to prove the SSC inequality below.
\end{remark}

\begin{lemma}\label{lemma:regstrongineq}
(SSC inequality for regularized problems) Under assumptions \textbf{A1,~A2} and \textbf{B1}, the problem in \EqRef{eq:regularizedprob} satisfies the SSC inequality with $\beta>0$ defined\footnote{$\beta$ still has the same form as in \EqRef{eq:betadef}, where each variable is accordingly changed from \EqRef{eq:constrainedprob} to \EqRef{eq:equivprob}.} in \EqRef{eq:betadef}.
\end{lemma}
\begin{proof}
By \LemmaRef{lemma:equivalent} (in \ApdRef{appendix:auxlemmas}), the optimal solution sets to \EqRef{eq:regularizedprob} and \EqRef{eq:equivprob} are
\begin{align}
\mathcal{W}^\star_r &=\{\mathbf{w}^\star:X\mathbf{w}^\star=\tilde{\mathbf{r}}^\star, r(\mathbf{w}^\star)=\tilde{s}^\star\}\label{eq:regoptset}\\
\text{and}~
\widetilde{\mathcal{W}}^\star &=\{(\mathbf{w}^\star,\varpi^\star):X\mathbf{w}^\star=\tilde{\mathbf{r}}^\star,r(\mathbf{w}^\star)\leq\varpi^\star=\tilde{s}^\star\}\label{eq:ineqoptset}\\
&=\{(\mathbf{w}^\star,\varpi^\star):X\mathbf{w}^\star=\tilde{\mathbf{r}}^\star, r(\mathbf{w}^\star)=\varpi^\star=\tilde{s}^\star\},\label{eq:eqoptset}
\end{align}
where $\tilde{\mathbf{r}}^\star$ and $\tilde{s}^\star$ are constants. By Eqs. (\ref{eq:regularizedprob}), (\ref{eq:equivprob}), for any $(\mathbf{w},\varpi)$ satisfying $r(\mathbf{w})=\varpi$, we have
\begin{align}
F(\mathbf{w})-F^\star &=\widetilde{F}(\mathbf{w},\varpi)-\widetilde{F}^\star,~\text{where}~\widetilde{F}^\star~\text{is the optimal objective function value of \EqRef{eq:equivprob}}.\nonumber
\end{align}
Considering Eqs. (\ref{eq:regoptset}), (\ref{eq:eqoptset}) together, we have
\begin{align}
\Pi_{\widetilde{\mathcal{W}}^\star}((\mathbf{w},\varpi))=(\Pi_{\mathcal{W}^\star_r}(\mathbf{w}),\tilde{s}^\star).\nonumber
\end{align}
Notice that $\widetilde{\mathcal{W}}^\star$ in \EqRef{eq:ineqoptset} is an intersection of a polyhedral set
$\{(\mathbf{w}^\star,\varpi^\star):r(\mathbf{w}^\star)\leq\varpi^\star\}$ and an affine space $\{(\mathbf{w}^\star,\varpi^\star):X\mathbf{w}^\star=\tilde{\mathbf{r}}^\star,\varpi^\star=\tilde{s}^\star\}$.
Thus, $(\mathbf{w},\varpi)\in\{(\mathbf{w},\varpi):r(\mathbf{w})\leq\varpi\}$ for any $\mathbf{w}$ with $r(\mathbf{w})=\varpi$.
Using a similar proof of \LemmaRef{lemma:strongineq}, we have for any finite $\mathbf{w}$ satisfying $r(\mathbf{w})=\varpi$:
\begin{align}
F(\mathbf{w})-F^\star &=\widetilde{F}(\mathbf{w},\varpi)-\widetilde{F}^\star\geq\frac{\beta}{2}\left\|(\mathbf{w},\varpi)-\Pi_{\widetilde{\mathcal{W}}^\star}((\mathbf{w},\varpi))\right\|^2\nonumber\\
&=\frac{\beta}{2}\left\|(\mathbf{w},\varpi)-(\Pi_{\mathcal{W}^\star_r}(\mathbf{w}),\tilde{s}^\star)\right\|^2\geq\frac{\beta}{2}\|\mathbf{w}-\Pi_{\mathcal{W}^\star_r}(\mathbf{w})\|^2.\nonumber
\end{align}
This completes the proof of the lemma.
\qed
\end{proof}

\begin{remark}\label{remark:sscineq}
The SSC inequality for both constrained and regularized problems is independent of algorithms. Thus, the SSC inequality may be of independent interest.
In particular, any algorithm whose linear convergence proof depends on $f(\mathbf{w})-f^\star\geq\mu\|\mathbf{w}-\mathbf{w}^\star\|^2$ ($\mu>0$) can potentially be adapted
to solve the non-strongly convex problems in Eqs. (\ref{eq:constrainedprob}), (\ref{eq:regularizedprob}) and achieve a linear convergence rate using the SSC inequality.
\end{remark}

\subsection{Proof Sketch of \ThemRef{theorem:linearconvergence} and \ThemRef{theorem:reglinearconvergence}}
Once we obtain the SSC inequality above, the proofs of \ThemRef{theorem:linearconvergence} and \ThemRef{theorem:reglinearconvergence} can be adapted from \cite{xiao2014proximal} (The key difference is that
we obtain the SSC inequality without the strong convexity). Due to
the space limit, we only provide a proof sketch of \ThemRef{theorem:linearconvergence} and the detailed proofs of both theorems are provided in \ApdRef{appendix:prooftheorems}.
\begin{proof}[Outline of the Proof of \ThemRef{theorem:linearconvergence}]
Let $\bar{\mathbf{w}}^{k}_{t}=\Pi_{\mathcal{W}^\star_c}(\mathbf{w}^k_t)$ for all $k,t\geq0$. Then we have $\bar{\mathbf{w}}^{k}_{t-1}\in\mathcal{W}^\star_c$,
which together with the definition of $\bar{\mathbf{w}}^{k}_{t}$ and $\mathbf{g}^k_{t}=(\mathbf{w}^k_{t-1}-\mathbf{w}^k_{t})/\eta$ implies that
\begin{align}
\left\|\mathbf{w}^k_t-\bar{\mathbf{w}}^{k}_{t}\right\|^2\leq\left\|\mathbf{w}^k_t-\bar{\mathbf{w}}^{k}_{t-1}\right\|^2=\left\|\mathbf{w}^k_{t-1}-\eta\mathbf{g}^k_{t}-\bar{\mathbf{w}}^{k}_{t-1}\right\|^2.\nonumber
\end{align}
Thus, following the proof of Theorem 3.1 in \cite{xiao2014proximal}, we obtain (the detailed proof is in \ApdRef{appendix:prooftheorems})
\begin{align}
&2\eta(1-4L_P\eta)\sum_{t=1}^{m}\Expect{\mathcal{F}^k_m}{f(\mathbf{w}_{t}^k)-f^{\star}}\nonumber\\
\leq & \Expect{\mathcal{F}^{k-1}_m}{\left\|\mathbf{w}^k_0-\bar{\mathbf{w}}^{k}_{0}\right\|^2}
+8L_P\eta^2(m+1)\Expect{\mathcal{F}^{k-1}_m}{f(\tilde{\mathbf{w}}^{k-1})-f^{\star}},\label{eq:objdiffineqmain}
\end{align}
By the convexity of $f(\cdot)$, we have
$f(\tilde{\mathbf{w}}^k)=f\left(\frac{1}{m}\sum_{t=1}^{m}\mathbf{w}_{t}^k\right)\leq\frac{1}{m}\sum_{t=1}^{m}f(\mathbf{w}_{t}^k)$.
Thus, we have
\begin{align}
m\left(f(\tilde{\mathbf{w}}^k)-f^{\star}\right)\leq\sum_{t=1}^{m}\left(f(\mathbf{w}_{t}^k)-f^{\star}\right),\label{eq:avgconvexityineqmain}
\end{align}
Considering \LemmaRef{lemma:strongineq} with bounded $\{\tilde{\mathbf{w}}^{k-1}\}$, $\tilde{\mathbf{w}}^{k-1}=\mathbf{w}^k_{0}\in\mathcal{W}$ and $\bar{\mathbf{w}}^k_{0}=\Pi_{\mathcal{W}^\star_c}(\mathbf{w}^k_{0})$, we have
\begin{align}
f(\tilde{\mathbf{w}}^{k-1})-f^{\star}=f(\mathbf{w}^k_{0})-f^{\star}\geq\frac{\beta}{2}\left\|\mathbf{w}^k_{0}-\bar{\mathbf{w}}^k_{0}\right\|^2,\nonumber
\end{align}
which together with Eqs. (\ref{eq:objdiffineqmain}), (\ref{eq:avgconvexityineqmain}) implies that
\begin{align}
2\eta(1-4L_P\eta)m\Expect{\mathcal{F}^k_m}{f(\tilde{\mathbf{w}}^k)-f^{\star}}
\leq\left(8L_P\eta^2(m+1)+\frac{2}{\beta}\right)\Expect{\mathcal{F}^{k-1}_m}{f(\tilde{\mathbf{w}}^{k-1})-f^{\star}}.\nonumber
\end{align}
Thus, we have
\begin{align}
\Expect{\mathcal{F}^k_m}{f(\tilde{\mathbf{w}}^k)-f^{\star}}\leq\left(\frac{4L_P\eta(m+1)}{(1-4L_P\eta)m}+\frac{1}{\beta\eta(1-4L_P\eta)m}\right)\Expect{\mathcal{F}^{k-1}_m}{f(\tilde{\mathbf{w}}^{k-1})-f^{\star}}.\nonumber
\end{align}
By considering the definition of $\rho$ in \EqRef{eq:linearrate}, we complete the proof of the theorem.
\qed
\end{proof}

\section{Experiments}\label{sec:experiments}
Empirical results in \cite{xiao2014proximal} have shown the effectiveness of Prox-SVRG on solving the regularized problem. Thus, in this section, we evaluate the effectiveness of VRPSG by solving the following $\ell_1$-constrained logistic regression problem:
\begin{align}
\min_{\mathbf{w}\in\mathbb{R}^d}\left\{f(\mathbf{w})=\frac{1}{n}\sum_{i=1}^n\log(1+\exp(-y_i\mathbf{x}_i^T\mathbf{w})),\quad s.t.~\|\mathbf{w}\|_1\leq\tau\right\},\nonumber
\end{align}
where $n$ is the number of samples; $\tau>0$ is the constrained parameter; $\mathbf{x}_i\in\mathbb{R}^d$ is the $i$-th sample; $y_i\in\{1,-1\}$ is the label of the sample $\mathbf{x}_i$.
For the above problem, it is easy to obtain that the convex component is $f_i(\mathbf{w})=\log(1+\exp(-y_i\mathbf{x}_i^T\mathbf{w}))$ and the Lipschitz constant of $\nabla f_i(\mathbf{w})$
is $\|\mathbf{x}_i\|^2/4$.

We conduct experiments on three real-world data sets: classic ($n=7094$, $d=41681$), reviews ($n=4069$, $d=18482$) and sports ($n=8580$, $d=14866$),
which are sparse text data and can be downloaded online\footnote{\scriptsize{\url{http://www.shi-zhong.com/software/docdata.zip}}}.
We conduct comparison by including the following algorithms (more experimental results are provided in the supplementary material due to space limit):
(1) AFG: the accelerated full gradient algorithm proposed in \cite{beck2009fast} with an adaptive line search.
(2) SGD: the stochastic gradient descent algorithm in \EqRef{eq:projstocgrad}. As suggested by \cite{duchi2009efficient}, we set the step size as $\eta_k=\eta_0/\sqrt{k}$, where $\eta_0$ is an initial step size.
(3) VRPSG: the variance-reduced projected stochastic gradient algorithm in this paper.
(4) VRPSG2: a hybrid algorithm by executing SGD for one pass over the data and then switching to the VRPSG algorithm (similar schemes are
also adopted in \cite{johnson2013accelerating,xiao2014proximal}).

\begin{figure*}[!ht]\vspace{-0.2cm}
\begin{minipage}[c]{1.0\linewidth}
\centering
\includegraphics[width=.3\linewidth]{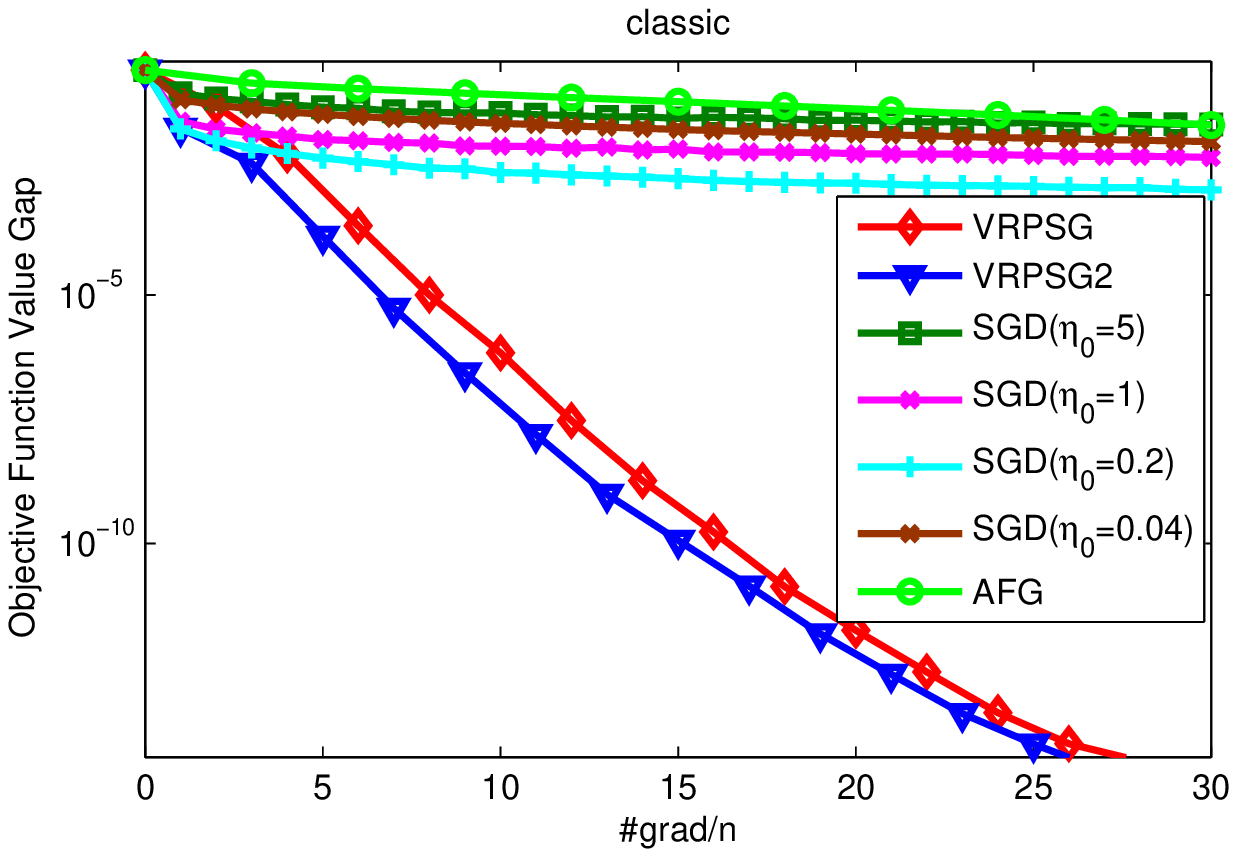}
\includegraphics[width=.3\linewidth]{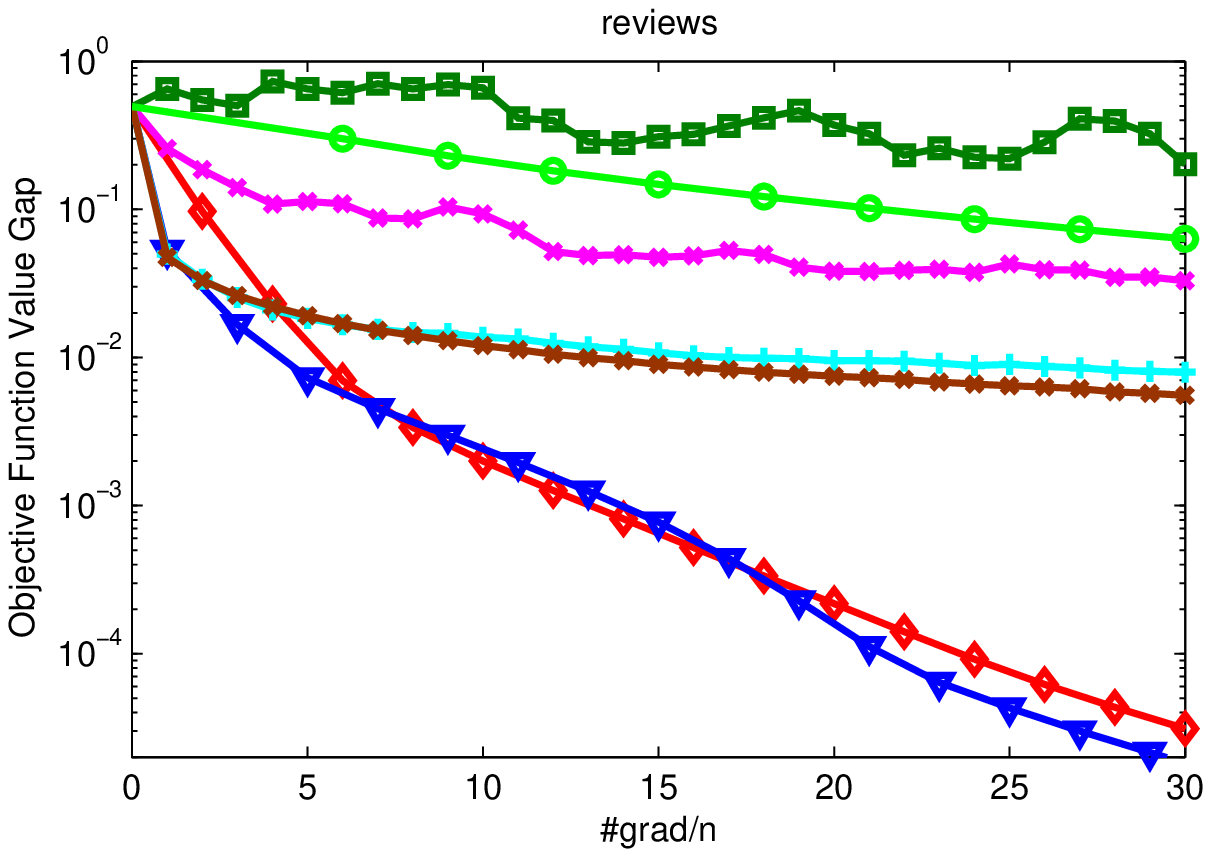}
\includegraphics[width=.3\linewidth]{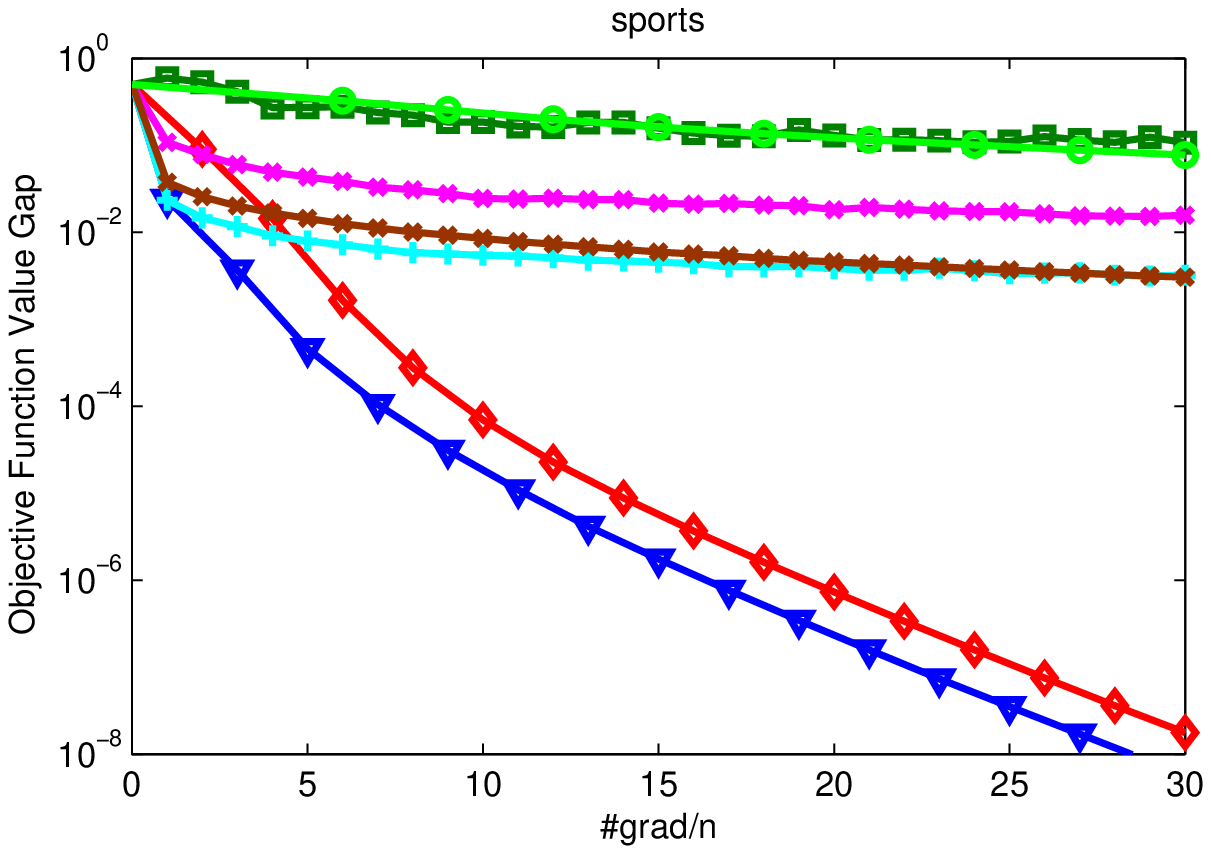}
\end{minipage}
\vskip 0.0cm
\vspace{-0.2cm}\caption{Comparison of different algorithms: the objective function value gap $f(\tilde{\mathbf{w}}^k)-f^\star$ vs. the number of gradient evaluations ($\sharp$grad/n) plots (averaged on 10 runs).
The parameter of VRPSG are set as $\tau=10$, $\eta=1/L_P$, $m=n$, $p_i=L_i/\sum_{i=1}^nL_i$; the step size of SGD is set as $\eta_k=\eta_0/\sqrt{k}$.}
\label{fig:compareobj}\vspace{-0.2cm}
\end{figure*}

Note that SGD is sensitive to the initial step size $\eta_0$ \cite{duchi2009efficient}. To have a fair comparison of different algorithms, we
set different values of $\eta_0$ for SGD to obtain the best performance ($\eta_0=5,1,0.2,0.04$).
To provide an implementation independent result for all algorithms, we report
the objective function value gap $f(\tilde{\mathbf{w}}^k)-f^\star$ vs. the number of gradient evaluations\footnote{Computing the gradient on a single sample counts as one gradient evaluation.}($\sharp$grad/n) plots
in \FigRef{fig:compareobj}, from which we have the following observations:
(a) Both stochastic algorithms (VRPSG and SGD with a proper initial step size) outperform the full gradient algorithm (AFG).
(b) SGD quickly decreases the objective function value in the beginning and gradually slows down in the following iterations.
In contrast, VRPSG decreases the objective function value linearly from the beginning.
This phenomenon is commonly expected due to the sub-linear convergence rate of SGD and the linear convergence rate of VRPSG.
(c) VRPSG2 performs slightly better than VRPSG, which demonstrates that the hybrid scheme can empirically improve
the performance. Similar results are also reported in \cite{johnson2013accelerating,xiao2014proximal}.

\section{Conclusion}\label{sec:conclusions}
In this paper, we study Prox-SVRG and its projected variant VRPSG on efficiently solving a class of non-strongly convex optimization problems
in both constrained and regularized settings.
Our main technical contribution is to establish a linear convergence analysis for both VRPSG and Prox-SVRG without strong convexity. To the best of our knowledge, this is the first linear convergence result for variance-reduced
stochastic gradient algorithms on solving both constrained and regularized problems without the strongly convex condition. In the future work, we will try to develop a more general convergence analysis for a wider range of problems including both non-polyhedral constrained and regularized optimization problems.


\small
\bibliography{nips2015}
\bibliographystyle{abbrv}

\normalsize
\newpage
\begin{center}
\large{\textbf{Supplementary Material for ``Linear Convergence of Variance-Reduced Stochastic Gradient without Strong Convexity''}}
\end{center}

\appendix
In this supplementary material, we first present some remarks for \ThemRef{theorem:linearconvergence} in \ApdRef{appendix:remarktheorem}. Then we present some auxiliary Lemmas in \ApdRef{appendix:auxlemmas}, which will be used in the proofs of linear convergence theorems in \ApdRef{appendix:prooftheorems}. Finally we present more experimental results in \ApdRef{appendix:moreresults}.

\section{Remarks for \ThemRef{theorem:linearconvergence}}\label{appendix:remarktheorem}

We have the following remarks on the convergence result in \ThemRef{theorem:linearconvergence}:
\begin{itemize}
\item Let $\eta=\gamma/L_P$ with $0<\gamma<1/4$. When $m$ is sufficiently large, we have
\begin{align}
\rho\approx\frac{ L_P/\beta}{\gamma(1-4\gamma)m}+\frac{4\gamma}{1-4\gamma},\nonumber
\end{align}
where $L_P/\beta$ can be treated as a \emph{pseudo condition number} of the problem in \EqRef{eq:constrainedprob}.
If we choose $\gamma=0.1$ and $m=100 L_P/\beta$, then $\rho\approx5/6$. Notice that at each outer iteration of \AlgRef{alg:projstocgrad},
$n+2m$ gradient evaluations (computing the gradient on a single sample counts as one gradient evaluation) are required. Thus, to obtain an $\epsilon$-accuracy solution (i.e., $\Expect{\mathcal{F}^k_m}{f(\tilde{\mathbf{w}}^k)-f^\star}\leq\epsilon$),
we need $O(n+L_P/\beta)\log(1/\epsilon)$ gradient evaluations by setting $m=\Theta(L_P/\beta)$. In particular, the complexity becomes
$O(n+L_{\mathrm{avg}}/\beta)\log(1/\epsilon)$ if we choose $p_i=L_i/\sum_{i=1}^nL_i$ for all $i\in\{1,\cdots,n\}$, and $O(n+L_{\max}/\beta )\log(1/\epsilon)$
if we choose $p_i=1/n$ for all $i\in\{1,\cdots,n\}$, where $L_{\mathrm{avg}}=\sum_{i=1}^nL_i/n$ and $L_{\max}=\max_{i\in\{1,\cdots,n\}}L_i$. Notice
that $L_{\mathrm{avg}}\leq L_{\max}$. Thus, sampling in proportion to the Lipschitz constant is better than sampling uniformly.

\item If $f$ is strongly convex with parameter $\tilde{\mu}$ and $p_i=L_i/\sum_{i=1}^nL_i$, VRPSG has the same complexity as Prox-SVRG \cite{xiao2014proximal}, that is,
VRPSG needs $O(n+L_{\mathrm{avg}}/\tilde{\mu})\log(1/\epsilon)$ gradient evaluations
to obtain an $\epsilon$-accuracy solution. In contrast, full gradient methods and standard stochastic gradient algorithms with diminishing step size $\eta_k=1/(\alpha k)$ require $O(nL/\tilde{\mu})\log(1/\epsilon)$
and $O(1/(\alpha \epsilon))$ gradient evaluations
to obtain a solution of the same accuracy.
Obviously, $O(n+L_{\mathrm{avg}}/\tilde{\mu})\log(1/\epsilon)$ and is far superior over $O(nL/\tilde{\mu})\log(1/\epsilon)$ and $O(1/(\alpha \epsilon))$
when the sample size $n$ and the condition number $L/\tilde{\mu}$ are very large.

\item If the Lipschitz constant $L_i$ is unknown and difficult to compute, we can use an upper bound $\hat{L}_i$ instead of $L_i$ to define $L_P=\max_{i\in\{1,\cdots,n\}}[\hat{L}_i/(np_i)]$
and the theorem still holds.

\item We can obtain a convergence rate with high probability. According to Markov's inequality with $f(\tilde{\mathbf{w}}^k)-f^\star\geq0$, \ThemRef{theorem:linearconvergence} implies that
\begin{align}
\Pr(f(\tilde{\mathbf{w}}^k)-f^\star\geq\epsilon)\leq\frac{\Expect{\mathcal{F}^k_m}{f(\tilde{\mathbf{w}}^k)-f^\star}}{\epsilon}\leq\frac{\rho^k(f(\tilde{\mathbf{w}}^0)-f^\star)}{\epsilon}.\nonumber
\end{align}
Therefore, we have $\Pr(f(\tilde{\mathbf{w}}^k)-f^\star\leq\epsilon)\geq1-\delta$, if $k\geq\log\left(\frac{f(\tilde{\mathbf{w}}^0)-f^\star}{\delta\epsilon}\right)/\log(1/\rho)$.
\end{itemize}

\section{Auxiliary Lemmas}\label{appendix:auxlemmas}
\begin{lemma}\label{lemma:lipschitzrelation}
Let $L$ and $L_i$ be the Lipschitz constants of $\nabla f(\mathbf{w})$ and $\nabla f_i(\mathbf{w})$, respectively.
Moreover, let $L_{\mathrm{avg}}=\sum_{i=1}^nL_i/n$, $L_{\max}=\max_{i\in\{1,\cdots,n\}}L_i$ and $L_P=\max_{i\in\{1,\cdots,n\}}[L_i/(np_i)]$ with $p_i\in(0,1),\sum_{i=1}^np_i=1$. Then we have
\begin{align}
L\leq L_{\mathrm{avg}}\leq L_P ~\mathrm{and}~L_{\mathrm{avg}}\leq L_{\max}.\nonumber
\end{align}
\end{lemma}
\begin{proof}
Based on the definition of Lipschitz continuity, we obtain that $L$ and $L_i$ are the smallest positive
constants such that for all $\mathbf{w},\mathbf{u}\in\mathbb{R}^d$:
\begin{align}
&\|\nabla f(\mathbf{w})-\nabla f(\mathbf{u})\|\leq L\|\mathbf{w}-\mathbf{u}\|, \label{eq:lipschitzineqf}\\
&\|\nabla f_i(\mathbf{w})-\nabla f_i(\mathbf{u})\|\leq L_i\|\mathbf{w}-\mathbf{u}\|.\label{eq:lipschitzineqfi}
\end{align}
Dividing \EqRef{eq:lipschitzineqfi} by $n$ and summing over $i=1,\cdots,n$, we have
\begin{align}
\frac{1}{n}\sum_{i=1}^n\|\nabla f_i(\mathbf{w})-\nabla f_i(\mathbf{u})\|\leq \frac{1}{n}\sum_{i=1}^nL_i\|\mathbf{w}-\mathbf{u}\|.\label{eq:sumlipschitzineq}
\end{align}
Based on the triangle inequality and $\nabla f(\mathbf{w})=\frac{1}{n}\sum_{i=1}^n\nabla f_i(\mathbf{w})$
we have
\begin{align}
\|\nabla f(\mathbf{w})-\nabla f(\mathbf{u})\|\leq \frac{1}{n}\sum_{i=1}^n\|\nabla f_i(\mathbf{w})-\nabla f_i(\mathbf{u})\|,\nonumber
\end{align}
which together with $L_{\mathrm{avg}}=\sum_{i=1}^nL_i/n$ and Eqs. (\ref{eq:lipschitzineqf}), (\ref{eq:sumlipschitzineq}) implies that $L\leq L_{\mathrm{avg}}$.

Define $\mathbf{s}=[L_1/p_1,\cdots,L_n/p_n]^T$. Noticing that
$L_P=\max_{i\in\{1,\cdots,n\}}[L_i/(np_i)]$ with $p_i\in(0,1),\sum_{i=1}^np_i=1$ and considering the definition of the dual norm, we have
\begin{align}
nL_P=\max_{i\in\{1,\cdots,n\}}\frac{L_i}{p_i}=\|\mathbf{s}\|_{\infty}=\sup_{\|\mathbf{t}\|_1\leq1}\mathbf{t}^T\mathbf{s}\geq \sum_{i=1}^np_i\frac{L_i}{p_i},\nonumber
\end{align}
which together with $L_{\mathrm{avg}}=\sum_{i=1}^nL_i/n$ immediately implies that $L_{\mathrm{avg}}\leq L_P$. $L_{\mathrm{avg}}\leq L_{\max}$
is obvious by the definition of $L_{\mathrm{avg}}=\sum_{i=1}^nL_i/n$ and $L_{\max}=\max_{i\in\{1,\cdots,n\}}L_i$.

\qed
\end{proof}

\begin{lemma}\label{lemma:graddiffbound}
Let $\mathbf{w}^\star\in\mathcal{W}^\star_c$ be any optimal solution to the problem in \EqRef{eq:constrainedprob}, $f^\star=f(\mathbf{w}^\star)$ be the optimal objective function value in \EqRef{eq:constrainedprob} and $L_P=\max_{i\in\{1,\cdots,n\}}[L_i/(np_i)]$ with $p_i\in(0,1),\sum_{i=1}^np_i=1$. Then under assumptions \textbf{A1}-\textbf{A3}, for all $\mathbf{w}\in\mathcal{W}$, we have
\begin{align}
\frac{1}{n}\sum_{i=1}^n\frac{1}{np_i}\left\|\nabla f_i(\mathbf{w})-\nabla f_i(\mathbf{w}^\star)\right\|^2\leq2L_P[f(\mathbf{w})-f^\star].\nonumber
\end{align}
\end{lemma}
\begin{proof}
For any $i\in\{1,\cdots,n\}$, we consider the following function
\begin{align}
\phi_i(\mathbf{w})=f_i(\mathbf{w})-f_i(\mathbf{w}^\star)-\nabla f_i(\mathbf{w}^\star)^T(\mathbf{w}-\mathbf{w}^\star).\nonumber
\end{align}
It follows from the convexity of $\phi_i(\mathbf{w})$ and $\nabla \phi_i(\mathbf{w}^\star)=\mathbf{0}$ that $\min_{\mathbf{w}\in\mathbb{R}^d}\phi_i(\mathbf{w})=\phi_i(\mathbf{w}^\star)=0$.
Recalling that $\nabla \phi_i(\mathbf{w})=\nabla f_i(\mathbf{w})-\nabla f_i(\mathbf{w}^\star)$ is $L_i$-Lipschitz continuous, we have for all $\mathbf{w}\in\mathcal{W}$:
\begin{align}
0=\phi_i(\mathbf{w}^\star)&\leq\min_{\eta\in\mathbb{R}} \phi_i(\mathbf{w}-\eta\nabla \phi_i(\mathbf{w}))\leq\min_{\eta\in\mathbb{R}}\left\{\phi_i(\mathbf{w})-\eta \|\nabla\phi_i(\mathbf{w})\|^2+\frac{L_i\eta^2}{2}\|\nabla\phi_i(\mathbf{w})\|^2\right\}\nonumber\\
&= \phi_i(\mathbf{w})-\frac{1}{2L_i}\|\nabla\phi_i(\mathbf{w})\|^2=\phi_i(\mathbf{w})-\frac{1}{2L_i}\|\nabla f_i(\mathbf{w})-\nabla f_i(\mathbf{w}^\star)\|^2,\nonumber
\end{align}
which implies for all $\mathbf{w}\in\mathcal{W}$:
\begin{align}
\|\nabla f_i(\mathbf{w})-\nabla f_i(\mathbf{w}^\star)\|^2\leq2L_i\phi_i(\mathbf{w})=2L_i(f_i(\mathbf{w})-f_i(\mathbf{w}^\star)-\nabla f_i(\mathbf{w}^\star)^T(\mathbf{w}-\mathbf{w}^\star)).\nonumber
\end{align}
Dividing the above inequality by $n^2p_i$ and summing over $i=1,\cdots,n$, we have
\begin{align}
\frac{1}{n}\sum_{i=1}^n\frac{1}{np_i}\|\nabla f_i(\mathbf{w})-\nabla f_i(\mathbf{w}^\star)\|^2\leq2L_P(f(\mathbf{w})-f(\mathbf{w}^\star)-\nabla f(\mathbf{w}^\star)^T(\mathbf{w}-\mathbf{w}^\star)),\label{eq:graddiffbound}
\end{align}
where we use $L_{\mathrm{avg}}=\sum_{i=1}^nL_i/n\leq L_P=\max_{i\in\{1,\cdots,n\}}[L_i/(np_i)]$ (see \LemmaRef{lemma:lipschitzrelation}) and $f(\mathbf{w})=\frac{1}{n}\sum_{i=1}^nf_i(\mathbf{w})$.
Recalling that $\mathbf{w}^\star\in\mathcal{W}^\star_c$ is an optimal solution to \EqRef{eq:constrainedprob} and $\mathbf{w}\in\mathcal{W}$, it follows from the optimality condition of \EqRef{eq:constrainedprob} that
\begin{align}
\nabla f(\mathbf{w}^\star)^T(\mathbf{w}-\mathbf{w}^\star)\geq0,\nonumber
\end{align}
which together with \EqRef{eq:graddiffbound} and $f^\star=f(\mathbf{w}^\star)$ immediately proves the lemma.
\qed
\end{proof}

\begin{lemma}\label{lemma:globalsolution}
Under assumptions \textbf{A1}-\textbf{A3}, for all $\mathbf{w}^\star\in\mathcal{W}^\star_c$, there exist unique $\mathbf{r}^\star$ and $s^\star$ such that
$X\mathbf{w}^\star=\mathbf{r}^\star$ and $\mathbf{q}^T\mathbf{w}^\star=s^\star$.
Moreover, $\mathcal{W}^\star_c=\{\mathbf{w}^\star: C\mathbf{w}^\star\leq\mathbf{b},~X\mathbf{w}^\star=\mathbf{r}^\star,~\mathbf{q}^T\mathbf{w}^\star=s^\star\}$.
\end{lemma}
\begin{proof}
By assumption \textbf{A1}, we know that $\mathcal{W}^\star_c$ is not empty. We first prove that there exists a unique $\mathbf{r}^\star$ such that
$X\mathbf{w}^\star=\mathbf{r}^\star$ by contradiction.
Assume that there are $\mathbf{w}^\star_1,\mathbf{w}^\star_2\in\mathcal{W}^\star_c$ such that $X\mathbf{w}^\star_1\neq X\mathbf{w}^\star_2$.
Then, the optimal objective function value is $f^\star=h(X\mathbf{w}^\star_1)+\mathbf{q}^T\mathbf{w}^\star_1=h(X\mathbf{w}^\star_2)+\mathbf{q}^T\mathbf{w}^\star_2$.
Due to $\mathbf{w}^\star_1,\mathbf{w}^\star_2\in\mathcal{W}^\star_c$ and the convexity of $\mathcal{W}^\star_c$, we have $(\mathbf{w}^\star_1+\mathbf{w}^\star_2)/2\in\mathcal{W}^\star_c$. Therefore,
\begin{align}
f^\star=h\left(\frac{1}{2}X\mathbf{w}^\star_1+\frac{1}{2}X\mathbf{w}^\star_2\right)+\frac{1}{2}\mathbf{q}^T(\mathbf{w}^\star_1+\mathbf{w}^\star_2).\label{eq:fstareqv}
\end{align}
On the other hand, the strong convexity of $h(\cdot)$ implies that
\begin{align}
h\left(\frac{1}{2}X\mathbf{w}^\star_1+\frac{1}{2}X\mathbf{w}^\star_2\right)
<\frac{1}{2}h(X\mathbf{w}^\star_1)+\frac{1}{2}h(X\mathbf{w}^\star_2),\nonumber
\end{align}
which together with \EqRef{eq:fstareqv} implies that
\begin{align}
f^\star<\frac{1}{2}h(X\mathbf{w}^\star_1)+\frac{1}{2}h(X\mathbf{w}^\star_2)+\frac{1}{2}\mathbf{q}^T(\mathbf{w}^\star_1+\mathbf{w}^\star_2)=f^\star,\nonumber
\end{align}
leading to a contradiction. Thus, there exists a unique $\mathbf{r}^\star$ such that for all $\mathbf{w}^\star\in\mathcal{W}^\star_c$, $X\mathbf{w}^\star=\mathbf{r}^\star$.
The uniqueness of $\mathbf{r}^\star$ and $f^\star$ immediately implies that there exists a unique $s^\star$ such that $\mathbf{q}^T\mathbf{w}^\star=s^\star$.

If $\mathbf{w}^\star\in\mathcal{W}^\star_c$, then $\mathbf{w}^\star\in\mathcal{W}$, $X\mathbf{w}^\star=\mathbf{r}^\star$ and $\mathbf{q}^T\mathbf{w}^\star=s^\star$, that is,
$\mathbf{w}^\star\in\{\mathbf{w}^\star: C\mathbf{w}^\star\leq\mathbf{b},~X\mathbf{w}^\star=\mathbf{r}^\star,~\mathbf{q}^T\mathbf{w}^\star=s^\star\}$ and hence
$\mathcal{W}^\star_c\subseteq\{\mathbf{w}^\star: C\mathbf{w}^\star\leq\mathbf{b},~X\mathbf{w}^\star=\mathbf{r}^\star,~\mathbf{q}^T\mathbf{w}^\star=s^\star\}$. If $\mathbf{w}^\star\in\{\mathbf{w}^\star: C\mathbf{w}^\star\leq\mathbf{b},~X\mathbf{w}^\star=\mathbf{r}^\star,~\mathbf{q}^T\mathbf{w}^\star=s^\star\}$, then
$\mathbf{w}^\star$ is a feasible solution and $f(\mathbf{w}^\star)=h(X\mathbf{w}^\star)+\mathbf{q}^T\mathbf{w}^\star=f^\star$,
that is, $\mathbf{w}^\star\in\mathcal{W}^\star_c$ and hence $\{\mathbf{w}^\star: C\mathbf{w}^\star\leq\mathbf{b},~X\mathbf{w}^\star=\mathbf{r}^\star,~\mathbf{q}^T\mathbf{w}^\star=s^\star\}\subseteq\mathcal{W}^\star_c$. Therefore, we have
$\mathcal{W}^\star_c=\{\mathbf{w}^\star: C\mathbf{w}^\star\leq\mathbf{b},~X\mathbf{w}^\star=\mathbf{r}^\star,~\mathbf{q}^T\mathbf{w}^\star=s^\star\}$.
\qed
\end{proof}

\begin{lemma}\label{lemma:boundedobjgap}
Let $\mathbf{w}\in\mathcal{W}=\{\mathbf{w}: X\mathbf{w}\leq\mathbf{b}\}$ and $f^\star$ be the optimal objective function value in \EqRef{eq:constrainedprob}.
Then under assumptions \textbf{A1}-\textbf{A3}, for any finite $\mathbf{w}$, there exists a constant $M>0$ such that
\begin{align}
f(\mathbf{w})-f^\star\leq M.\nonumber
\end{align}
\end{lemma}
\begin{proof}
For any $\mathbf{w}^\star\in\mathcal{W}^\star_c$, we have $f^\star=f(\mathbf{w}^\star)$. Recalling that $\nabla f(\mathbf{w})$ is Lipschitz continuous with constant $L$, we have
\begin{align}
f(\mathbf{w})-f^\star\leq \nabla f(\mathbf{w}^\star)^T(\mathbf{w}-\mathbf{w}^\star) + \frac{L}{2}\|\mathbf{w}-\mathbf{w}^\star\|^2,\nonumber
\end{align}
which together with \LemmaRef{lemma:globalsolution} implies that there exists a constant vector $\mathbf{r}^\star$ such that
\begin{align}
f(\mathbf{w})-f^\star &\leq (X^T\nabla h(\mathbf{r}^\star)+\mathbf{q})^T(\mathbf{w}-\mathbf{w}^\star) + \frac{L}{2}\|\mathbf{w}-\mathbf{w}^\star\|^2\nonumber\\
&\leq\|X^T\nabla h(\mathbf{r}^\star)+\mathbf{q}\|\|\mathbf{w}-\mathbf{w}^\star\| + \frac{L}{2}\|\mathbf{w}-\mathbf{w}^\star\|^2,\nonumber
\end{align}
Recall that $\|X^T\nabla h(\mathbf{r}^\star)+\mathbf{q}\|$ is constant and both $\mathbf{w}$ and $\mathbf{w}^\star$ are finite. Thus,
the right-hand-side of the above inequality must be upper bounded by a positive constant $M$. This completes the proof of the lemma.
\qed
\end{proof}

\begin{lemma}\label{lemma:hoffmanbound}
(Hoffman's bound, Lemma 4.3 \cite{wang2014iteration}) Let $\mathcal{V}=\{\mathbf{w}: C\mathbf{w}\leq\mathbf{b},~X\mathbf{w}=\mathbf{r}\}$
be a non-empty polyhedron. Then for any $\mathbf{w}\in\mathbb{R}^d$, there exist a feasible point $\mathbf{w}^\star$ of $\mathcal{V}$
and a constant $\theta>0$ such that
\begin{align}
\|\mathbf{w}-\mathbf{w}^\star\|\leq\theta\left\|
\begin{array}{c}
[C\mathbf{w}-\mathbf{b}]^+\\
X\mathbf{w}-\mathbf{r}
\end{array}
\right\|,\nonumber
\end{align}
where $[C\mathbf{w}-\mathbf{b}]^+$ denotes the Euclidean projection of $C\mathbf{w}-\mathbf{b}$ onto
the non-negative orthant and $\theta$ is the Hoffman constant defined by
\begin{align}
\theta=\sup_{\mathbf{u},\mathbf{v}}\left\{\left\|
\begin{array}{c}
\mathbf{u}\\
\mathbf{v}
\end{array}
\right\|:
\begin{array}{l}
\|C^T\mathbf{u}+X^T\mathbf{v}\|=1,~\mathbf{u}\geq0\\
the~rows~of~C~and~X~corresponding~to~the~nonzero\\
components~of~\mathbf{u}~and~\mathbf{v}~are~linearly~independent\\
\end{array}
\right\}<\infty.\nonumber
\end{align}
\end{lemma}

\begin{remark}\label{remark:hoffmanbound}
Let $\mathcal{D}$ be a set including all matrices formed by the linearly independent columns of the matrix $[C^T,X^T]$. Then for all $D\in\mathcal{D}$,
$D^TD$ is invertible and we have
\begin{align}
0<\theta\leq\max_{D\in\mathcal{D}}\sigma_{\max}((D^TD)^{-1}D^T),~\text{where}~\sigma_{\max}(\cdot)~\text{denotes the maximum singular value}.\nonumber
\end{align}
\end{remark}

\begin{lemma}\label{lemma:equivalent}
(a) \EqRef{eq:regularizedprob} is equivalent to \EqRef{eq:equivprob}. Specifically, if $\mathbf{w}^\star$ is an optimal solution to \EqRef{eq:regularizedprob},
then $(\mathbf{w}^\star,r(\mathbf{w}^\star))$ must be an optimal solution to \EqRef{eq:equivprob}. If
$(\mathbf{w}^\star,\varpi^\star)$ is an optimal solution to \EqRef{eq:equivprob}, then $\varpi^\star=r(\mathbf{w}^\star)$ must hold and
$\mathbf{w}^\star$ must be an optimal solution to \EqRef{eq:regularizedprob}. (b) There exist constant $\tilde{\mathbf{r}}^\star$ and $\tilde{s}^\star$ such that the optimal solution sets to \EqRef{eq:regularizedprob} and \EqRef{eq:equivprob} are
$\mathcal{W}^\star_r=\{\mathbf{w}^\star:X\mathbf{w}^\star=\tilde{\mathbf{r}}^\star, r(\mathbf{w}^\star)=\tilde{s}^\star\}$ and
$\widetilde{\mathcal{W}}^\star=\{(\mathbf{w}^\star,\varpi^\star):X\mathbf{w}^\star=\tilde{\mathbf{r}}^\star, r(\mathbf{w}^\star)\leq\varpi^\star=\tilde{s}^\star\}=\{(\mathbf{w}^\star,\varpi^\star):X\mathbf{w}^\star=\tilde{\mathbf{r}}^\star, r(\mathbf{w}^\star)=\varpi^\star=\tilde{s}^\star\}$.
\end{lemma}
\begin{proof}
(a) By the optimality condition of \EqRef{eq:regularizedprob}, $\mathbf{w}^\star$ is an optimal solution to \EqRef{eq:regularizedprob}, if and only if
\begin{align}\label{eq:iffeqreg}
\nabla f(\mathbf{w}^\star)^T(\mathbf{w}-\mathbf{w}^\star)+r(\mathbf{w})-r(\mathbf{w}^\star)\geq0,~\forall\mathbf{w}\in\mathbb{R}^d.
\end{align}
By the optimality condition of \EqRef{eq:equivprob}, $(\mathbf{w}^\star,\varpi^\star)$ is an optimal solution to \EqRef{eq:equivprob}, if and only if
\begin{align}\label{eq:iffeqequiv}
\nabla f(\mathbf{w}^\star)^T(\mathbf{w}-\mathbf{w}^\star)+\varpi-\varpi^\star\geq0,~\forall \mathbf{w},\varpi~\text{such that}~r(\mathbf{w})\leq \varpi.
\end{align}
If $\mathbf{w}^\star$ is an optimal solution to \EqRef{eq:regularizedprob}, then \EqRef{eq:iffeqreg} immediately implies \EqRef{eq:iffeqequiv} by setting $\varpi^\star=r(\mathbf{w}^\star)$,
i.e., $(\mathbf{w}^\star,r(\mathbf{w}^\star))$ must be an optimal solution to \EqRef{eq:equivprob}. If $(\mathbf{w}^\star,\varpi^\star)$ is an optimal solution to \EqRef{eq:equivprob}, let us
assume that $\varpi^\star>r(\mathbf{w}^\star)$. Then we have $\widetilde{F}(\mathbf{w}^\star,\varpi^\star)>\widetilde{F}(\mathbf{w}^\star,r(\mathbf{w}^\star))$, which contradicts the fact that
$(\mathbf{w}^\star,\varpi^\star)$ is an optimal solution to \EqRef{eq:equivprob}. Therefore, $\varpi^\star=r(\mathbf{w}^\star)$ must hold. Moreover, \EqRef{eq:iffeqequiv} immediately implies \EqRef{eq:iffeqreg} by setting and $\varpi=r(\mathbf{w})$ and considering that $\varpi^\star=r(\mathbf{w}^\star)$, i.e., $\mathbf{w}^\star$ must be an optimal solution to \EqRef{eq:regularizedprob}.

(b) Recalling that $r(\mathbf{w})$ is convex, we can use a similar argument of \LemmaRef{lemma:globalsolution} (in \ApdRef{appendix:auxlemmas}) to show that there exist constants $\tilde{\mathbf{r}}^\star_1,\tilde{\mathbf{r}}^\star_2,\tilde{s}^\star_1,\tilde{s}^\star_2$ such that the optimal solution sets to \EqRef{eq:regularizedprob} and \EqRef{eq:equivprob} are $\mathcal{W}^\star_r=\{\mathbf{w}^\star:X\mathbf{w}^\star=\tilde{\mathbf{r}}^\star_1, r(\mathbf{w}^\star)=\tilde{s}^\star_1\}$ and
$\widetilde{\mathcal{W}}^\star=\{(\mathbf{w}^\star,\varpi^\star):X\mathbf{w}^\star_2=\tilde{\mathbf{r}}^\star_2, r(\mathbf{w}^\star)\leq\varpi^\star=\tilde{s}^\star_2\}$. By the equivalence in (a), there exist $\tilde{\mathbf{r}}^\star$ and $\tilde{s}^\star$ such that
$\tilde{\mathbf{r}}^\star=\tilde{\mathbf{r}}^\star_1=\tilde{\mathbf{r}}^\star_2$ and $\tilde{s}^\star=\tilde{s}^\star_1=\tilde{s}^\star_2$. Next, we prove that $\{(\mathbf{w}^\star,\varpi^\star):X\mathbf{w}^\star=\tilde{\mathbf{r}}^\star, r(\mathbf{w}^\star)\leq\varpi^\star=\tilde{s}^\star\}=\{(\mathbf{w}^\star,\varpi^\star):X\mathbf{w}^\star=\tilde{\mathbf{r}}^\star, r(\mathbf{w}^\star)=\varpi^\star=\tilde{s}^\star\}$. To show this, we only need to prove that
$\{(\mathbf{w}^\star,\varpi^\star):X\mathbf{w}^\star=\tilde{\mathbf{r}}^\star, r(\mathbf{w}^\star)\leq\varpi^\star=\tilde{s}^\star\}\subseteq\{(\mathbf{w}^\star,\varpi^\star):X\mathbf{w}^\star=\tilde{\mathbf{r}}^\star, r(\mathbf{w}^\star)=\varpi^\star=\tilde{s}^\star\}$,
since $\{(\mathbf{w}^\star,\varpi^\star):X\mathbf{w}^\star=\tilde{\mathbf{r}}^\star, r(\mathbf{w}^\star)=\varpi^\star=\tilde{s}^\star\}\subseteq\{(\mathbf{w}^\star,\varpi^\star):X\mathbf{w}^\star=\tilde{\mathbf{r}}^\star, r(\mathbf{w}^\star)\leq\varpi^\star=\tilde{s}^\star\}$ holds trivially. Let $(\mathbf{w}^\star,\varpi^\star)\in\{(\mathbf{w}^\star,\varpi^\star):X\mathbf{w}^\star=\tilde{\mathbf{r}}^\star, r(\mathbf{w}^\star)\leq\varpi^\star=\tilde{s}^\star\}$, i.e., $(\mathbf{w}^\star,\varpi^\star)$ is an optimal solution to \EqRef{eq:equivprob}. Then, by the conclusion in (a), we have $r(\mathbf{w}^\star)=\varpi^\star=\tilde{s}^\star$, which immediately implies that $(\mathbf{w}^\star,\varpi^\star)\in\{(\mathbf{w}^\star,\varpi^\star):X\mathbf{w}^\star=\tilde{\mathbf{r}}^\star, r(\mathbf{w}^\star)=\varpi^\star=\tilde{s}^\star\}$. Therefore, we have $\{(\mathbf{w}^\star,\varpi^\star):X\mathbf{w}^\star=\tilde{\mathbf{r}}^\star, r(\mathbf{w}^\star)\leq\varpi^\star=\tilde{s}^\star\}\subseteq\{(\mathbf{w}^\star,\varpi^\star):X\mathbf{w}^\star=\tilde{\mathbf{r}}^\star, r(\mathbf{w}^\star)=\varpi^\star=\tilde{s}^\star\}$.
\qed
\end{proof}

\section{Proofs of Linear Convergence Theorems}\label{appendix:prooftheorems}
In addition to the SSC inequalities in Lemmas \ref{lemma:strongineq}, \ref{lemma:regstrongineq}, we also need the following two lemmas (Lemmas \ref{lemma:stocgradvar}, \ref{lemma:interbound}) to prove the linear convergence theorems.
Note that Lemmas \ref{lemma:stocgradvar}, \ref{lemma:interbound}
are established for constrained optimization problems which are adapted from Corollary 3.5 and Lemma 3.7
for regularized optimization problems in \cite{xiao2014proximal}.

The first lemma bounds the variance of $\mathbf{v}_{t}^k$ in terms of the difference of objective functions.
\begin{lemma}\label{lemma:stocgradvar}
Let $\mathbf{w}^\star\in\mathcal{W}^\star_c$ be any optimal solution to the problem in \EqRef{eq:constrainedprob}, $f^\star=f(\mathbf{w}^\star)$ be the optimal objective function value in \EqRef{eq:constrainedprob}. Then under
assumptions \textbf{A1}-\textbf{A3}, we have
\begin{align}
&\Expect{\mathcal{F}^k_t}{\mathbf{v}_{t}^k\mid\mathcal{F}^{k}_{t-1}}=\nabla f(\mathbf{w}^k_{t-1}),\label{eq:stocgradmean}\\
&\Expect{\mathcal{F}^k_t}{\left\|\mathbf{v}_{t}^k-\nabla f(\mathbf{w}^k_{t-1})\right\|^2\mid\mathcal{F}^{k}_{t-1}}\leq4L_P\left(f(\mathbf{w}_{t-1}^k)-f^{\star}+f(\tilde{\mathbf{w}}^{k-1})-f^{\star}\right),\label{eq:stocgradvar}
\end{align}
where $\mathcal{F}^k_t$ is defined in \ThemRef{theorem:linearconvergence}; $\mathbf{v}_{t}^k, \mathbf{w}^k_{t-1}, \tilde{\mathbf{w}}^{k-1}$ are defined in \AlgRef{alg:projstocgrad};
$L_P=\max_{i\in\{1,\cdots,n\}}[L_i/(np_i)]$.
\end{lemma}
\begin{proof}
Taking expectation with respect to $\mathcal{F}^{k}_{t}$ conditioned on $\mathcal{F}^{k}_{t-1}$ and noticing that $\mathcal{F}^{k}_{t}=\mathcal{F}^{k}_{t-1}\cup\{i^k_t\}$, we have
\begin{align}
&\Expect{\mathcal{F}^k_t}{\frac{1}{np_{i^k_t}}\nabla f_{i^k_t}(\mathbf{w}^k_{t-1})\mid\mathcal{F}^{k}_{t-1}}=\sum_{i=1}^n\frac{p_{i}}{np_{i}}\nabla f_{i}(\mathbf{w}^k_{t-1})=\nabla f(\mathbf{w}^k_{t-1}),\nonumber\\
&\Expect{\mathcal{F}^k_t}{\frac{1}{np_{i^k_t}}\nabla f_{i^k_t}(\tilde{\mathbf{w}}^{k-1})\mid\mathcal{F}^{k}_{t-1}}=\sum_{i=1}^n\frac{p_{i}}{np_{i}}\nabla f_{i}(\tilde{\mathbf{w}}^{k-1})=\nabla f(\tilde{\mathbf{w}}^{k-1}).\nonumber
\end{align}
It follows that
\begin{align}
\Expect{\mathcal{F}^k_t}{\mathbf{v}_{t}^k\mid\mathcal{F}^{k}_{t-1}}=\Expect{\mathcal{F}^k_t}{\frac{1}{np_{i^k_t}}(\nabla f_{i^k_t}(\mathbf{w}^k_{t-1})-\nabla f_{i^k_t}(\tilde{\mathbf{w}}^{k-1}))+\nabla f(\tilde{\mathbf{w}}^{k-1})\mid\mathcal{F}^{k}_{t-1}}=\nabla f(\mathbf{w}^k_{t-1}).\nonumber
\end{align}
We next prove \EqRef{eq:stocgradvar} as follows:
\begin{align}
&\Expect{\mathcal{F}^k_t}{\left\|\mathbf{v}_{t}^k-\nabla f(\mathbf{w}^k_{t-1})\right\|^2\mid\mathcal{F}^{k}_{t-1}}\nonumber\\
=&\Expect{\mathcal{F}^k_t}{\left\|\frac{1}{np_{i^k_t}}\left(\nabla f_{i^k_t}(\mathbf{w}_{t-1}^k)-\nabla f_{i^k_t}(\tilde{\mathbf{w}}^{k-1})\right)-\left(\nabla f(\mathbf{w}_{t-1}^{k})-\nabla f(\tilde{\mathbf{w}}^{k-1})\right)\right\|^2\mid\mathcal{F}^{k}_{t-1}}\nonumber\\
=&\Expect{\mathcal{F}^k_t}{\left\|\frac{1}{np_{i^k_t}}\left(\nabla f_{i^k_t}(\mathbf{w}_{t-1}^k)-\nabla f_{i^k_t}(\tilde{\mathbf{w}}^{k-1})\right)\right\|^2\mid\mathcal{F}^{k}_{t-1}}-\left\|\nabla f(\mathbf{w}_{t-1}^{k})-\nabla f(\tilde{\mathbf{w}}^{k-1})\right\|^2\nonumber\\
\leq&\Expect{\mathcal{F}^k_t}{\left\|\frac{1}{np_{i^k_t}}\left(\nabla f_{i^k_t}(\mathbf{w}_{t-1}^k)-\nabla f_{i^k_t}(\tilde{\mathbf{w}}^{k-1})\right)\right\|^2\mid\mathcal{F}^{k}_{t-1}}\nonumber\\
\leq&2\Expect{\mathcal{F}^k_t}{\left\|\frac{1}{np_{i^k_t}}\left(\nabla f_{i^k_t}(\mathbf{w}_{t-1}^k)-\nabla f_{i^k_t}(\mathbf{w}^{\star})\right)\right\|^2\mid\mathcal{F}^{k}_{t-1}}\nonumber\\
&+2\Expect{\mathcal{F}^k_t}{\left\|\frac{1}{np_{i^k_t}}\left(\nabla f_{i^k_t}(\tilde{\mathbf{w}}^{k-1})-\nabla f_{i^k_t}(\mathbf{w}^\star)\right)\right\|^2\mid\mathcal{F}^{k}_{t-1}}\nonumber\\
=&2\sum_{i=1}^n\frac{p_i}{(np_{i})^2}\left\|\nabla f_{i}(\mathbf{w}_{t-1}^k)-\nabla f_{i}(\mathbf{w}^{\star})\right\|^2+2\sum_{i=1}^n\frac{p_i}{(np_{i})^2}\left\|\nabla f_{i}(\tilde{\mathbf{w}}^{k-1})-\nabla f_{i}(\mathbf{w}^{\star})\right\|^2\nonumber\\
\leq&4L_P\left(f(\mathbf{w}_{t-1}^k)-f(\mathbf{w}^{\star})+f(\tilde{\mathbf{w}}^{k-1})-f(\mathbf{w}^{\star})\right)\nonumber\\
=&4L_P\left(f(\mathbf{w}_{t-1}^k)-f^{\star}+f(\tilde{\mathbf{w}}^{k-1})-f^{\star}\right),\nonumber
\end{align}
where the second equality is due to
\begin{align}
\Expect{\mathcal{F}^k_t}{\frac{1}{np_{i^k_t}}\left(\nabla f_{i^k_t}(\mathbf{w}_{t-1}^k)-\nabla f_{i^k_t}(\tilde{\mathbf{w}}^{k-1})\right)\mid\mathcal{F}^{k}_{t-1}}=\nabla f(\mathbf{w}^{k})-\nabla f(\tilde{\mathbf{w}}^{k-1})\nonumber
\end{align}
and $\Expectg{\|\bm{\xi}-\Expectg{\bm{\xi}}\|^2}=\Expectg{\|\bm{\xi}\|^2}-\|\Expectg{\bm{\xi}}\|^2$ for all random vector $\bm{\xi}\in\mathbb{R}^d$;
the second inequality is due to $\|\mathbf{x}+\mathbf{y}\|^2\leq2\|\mathbf{x}\|^2+2\|\mathbf{y}\|^2$;
the third inequality is due to \LemmaRef{lemma:graddiffbound} with $\mathbf{w}_{t-1}^k,\tilde{\mathbf{w}}^{k-1}\in\mathcal{W}$, where $\mathbf{w}_{t-1}^k\in\mathcal{W}$ is obvious
and $\tilde{\mathbf{w}}^{k-1}\in\mathcal{W}$ follows from the fact that $\tilde{\mathbf{w}}^{k-1}$ is a convex combination of vectors in the convex set $\mathcal{W}$.
\qed
\end{proof}

The second lemma presents a bound independent of the algorithm. The terms in the left-hand side of the bound will appear in the proof of \ThemRef{theorem:linearconvergence}.
\begin{lemma}\label{lemma:interbound}
Let $\mathbf{w}^\star\in\mathcal{W}^\star_c$ be any optimal solution to the problem in \EqRef{eq:constrainedprob}, $f^\star=f(\mathbf{w}^\star)$ be the optimal objective function value in \EqRef{eq:constrainedprob}, $\bm{\delta}_{t}^k=\nabla f(\mathbf{w}^k_{t-1})-\mathbf{v}^k_{t}$, $\mathbf{g}^k_{t}=(\mathbf{w}^k_{t-1}-\mathbf{w}^k_{t})/\eta$ and $0<\eta\leq1/L$. Then we have
\begin{align}
\left(\mathbf{w}^\star-\mathbf{w}^k_{t-1}\right)^T\mathbf{g}^k_{t}+\frac{\eta}{2}\left\|\mathbf{g}^k_{t}\right\|^2\leq f^\star-f(\mathbf{w}^k_{t})
-\left(\mathbf{w}^\star-\mathbf{w}^k_{t}\right)^T\bm{\delta}_{t}^k.\nonumber
\end{align}
\end{lemma}

\begin{proof}
We know that $\mathbf{w}^\star\in\mathcal{W}^\star_c\subseteq\mathcal{W}$. Thus, by the optimality condition of
$\mathbf{w}_{t}^{k}=\Pi_{\mathcal{W}}(\mathbf{w}_{t-1}^{k}-\eta\mathbf{v}_{t}^k)=\Argmin_{\mathbf{w}\in\mathcal{W}}\frac{1}{2}\|\mathbf{w}-(\mathbf{w}_{t-1}^{k}-\eta\mathbf{v}_{t}^k)\|^2$, we have
\begin{align}
(\mathbf{w}_{t}^{k}-\mathbf{w}_{t-1}^{k}+\eta\mathbf{v}_{t}^k)^T(\mathbf{w}^\star-\mathbf{w}_{t}^{k})\geq0,\nonumber
\end{align}
which together with $\mathbf{g}^k_{t}=(\mathbf{w}^k_{t-1}-\mathbf{w}^k_{t})/\eta$ implies that
\begin{align}
(\mathbf{w}^\star-\mathbf{w}_{t}^{k})^T\mathbf{v}_{t}^k\geq (\mathbf{w}^\star-\mathbf{w}_{t}^{k})^T\mathbf{g}^k_{t}.\label{eq:projoptcondition}
\end{align}
By the convexity of $f(\cdot)$, we have
\begin{align}
f(\mathbf{w}^\star)\geq f(\mathbf{w}_{t-1}^{k})+\nabla f(\mathbf{w}_{t-1}^{k})^T(\mathbf{w}^\star-\mathbf{w}_{t-1}^{k}).\label{eq:convexityineq}
\end{align}
Recalling that $f(\cdot)$ is $L$-Lipschitz continuous gradient, we have
\begin{align}
f(\mathbf{w}_{t-1}^{k})\geq f(\mathbf{w}_{t}^{k})-\nabla f(\mathbf{w}_{t-1}^{k})^T(\mathbf{w}_{t}^{k}-\mathbf{w}_{t-1}^{k})-\frac{L}{2}\left\|\mathbf{w}^k_{t}-\mathbf{w}^k_{t-1}\right\|^2,\nonumber
\end{align}
which together with \EqRef{eq:convexityineq} implies that
\begin{align}
f(\mathbf{w}^\star)\geq &f(\mathbf{w}_{t}^{k})-\nabla f(\mathbf{w}_{t-1}^{k})^T(\mathbf{w}_{t}^{k}-\mathbf{w}_{t-1}^{k})-\frac{L}{2}\left\|\mathbf{w}^k_{t}-\mathbf{w}^k_{t-1}\right\|^2+\nabla f(\mathbf{w}_{t-1}^{k})^T(\mathbf{w}^\star-\mathbf{w}_{t-1}^{k})\nonumber\\
= &f(\mathbf{w}_{t}^{k})+\nabla f(\mathbf{w}_{t-1}^{k})^T(\mathbf{w}^\star-\mathbf{w}_{t}^{k})-\frac{L\eta^2}{2}\left\|\mathbf{g}^k_{t}\right\|^2\nonumber\\
=& f(\mathbf{w}_{t}^{k})+(\mathbf{w}^\star-\mathbf{w}_{t}^{k})^T\bm{\delta}_{t}^k+(\mathbf{w}^\star-\mathbf{w}_{t}^{k})^T\mathbf{v}_{t}^{k}-\frac{L\eta^2}{2}\left\|\mathbf{g}^k_{t}\right\|^2\nonumber\\
\geq& f(\mathbf{w}_{t}^{k})+(\mathbf{w}^\star-\mathbf{w}_{t}^{k})^T\bm{\delta}_{t}^k+(\mathbf{w}^\star-\mathbf{w}_{t}^{k})^T\mathbf{g}^k_{t}-\frac{L\eta^2}{2}\left\|\mathbf{g}^k_{t}\right\|^2\nonumber\\
=&f(\mathbf{w}_{t}^{k})+(\mathbf{w}^\star-\mathbf{w}_{t}^{k})^T\bm{\delta}_{t}^k+(\mathbf{w}^\star-\mathbf{w}_{t-1}^{k}+\mathbf{w}_{t-1}^{k}-\mathbf{w}_{t}^{k})^T\mathbf{g}^k_{t}-\frac{L\eta^2}{2}\left\|\mathbf{g}^k_{t}\right\|^2\nonumber\\
=&f(\mathbf{w}_{t}^{k})+(\mathbf{w}^\star-\mathbf{w}_{t}^{k})^T\bm{\delta}_{t}^k+(\mathbf{w}^\star-\mathbf{w}_{t-1}^{k})^T\mathbf{g}^k_{t}+\frac{\eta}{2}(2-L\eta)\left\|\mathbf{g}^k_{t}\right\|^2\nonumber\\
\geq &f(\mathbf{w}_{t}^{k})+(\mathbf{w}^\star-\mathbf{w}_{t}^{k})^T\bm{\delta}_{t}^k+(\mathbf{w}^\star-\mathbf{w}_{t-1}^{k})^T\mathbf{g}^k_{t}+\frac{\eta}{2}\left\|\mathbf{g}^k_{t}\right\|^2,\nonumber
\end{align}
where the first and fourth equalities are due to $\mathbf{g}^k_{t}=(\mathbf{w}^k_{t-1}-\mathbf{w}^k_{t})/\eta$; the second equality is due to $\bm{\delta}_{t}^k=\nabla f(\mathbf{w}^k_{t-1})-\mathbf{v}^k_{t}$;
the second inequality is due to \EqRef{eq:projoptcondition}; the last inequality is due to $0<\eta\leq1/L$.
Rearranging the above inequality by noticing that $f^\star=f(\mathbf{w}^\star)$, we prove the lemma.
\qed
\end{proof}

Based on Lemmas \ref{lemma:strongineq}, \ref{lemma:stocgradvar}, \ref{lemma:interbound}, we are now ready to complete the proof of \ThemRef{theorem:linearconvergence} as follows:
\begin{proof}[Proof of \ThemRef{theorem:linearconvergence}]
Let $\bar{\mathbf{w}}^{k}_{t}=\Pi_{\mathcal{W}^\star_c}(\mathbf{w}^k_t)$ for all $k,t\geq0$. Then we have $\bar{\mathbf{w}}^{k}_{t-1}\in\mathcal{W}^\star_c$,
which together with the definition of $\bar{\mathbf{w}}^{k}_{t}$ and $\mathbf{g}^k_{t}=(\mathbf{w}^k_{t-1}-\mathbf{w}^k_{t})/\eta$ implies that
\begin{align}
&\left\|\mathbf{w}^k_t-\bar{\mathbf{w}}^{k}_{t}\right\|^2\leq\left\|\mathbf{w}^k_t-\bar{\mathbf{w}}^{k}_{t-1}\right\|^2=\left\|\mathbf{w}^k_{t-1}-\eta\mathbf{g}^k_{t}-\bar{\mathbf{w}}^{k}_{t-1}\right\|^2\nonumber\\
&=\left\|\mathbf{w}^k_{t-1}-\bar{\mathbf{w}}^{k}_{t-1}\right\|^2+2\eta(\bar{\mathbf{w}}^{k}_{t-1}-\mathbf{w}^k_{t-1})^T\mathbf{g}^k_{t}+\eta^2\left\|\mathbf{g}^k_{t}\right\|^2\nonumber\\
&\leq\left\|\mathbf{w}^k_{t-1}-\bar{\mathbf{w}}^{k}_{t-1}\right\|^2+2\eta\left(f^\star-f(\mathbf{w}^k_{t})
-(\bar{\mathbf{w}}^{k}_{t-1}-\mathbf{w}^k_{t})^T\bm{\delta}_{t}^k\right),\label{eq:iteratediff}
\end{align}
where the last inequality is due to \LemmaRef{lemma:interbound} with $\bar{\mathbf{w}}^{k}_{t-1}\in\mathcal{W}^\star_c$
and $0<\eta<1/(4L_P)<1/(L_P)\leq1/L$ (see \LemmaRef{lemma:lipschitzrelation} in the supplementary material).
To bound the quantity $-(\bar{\mathbf{w}}^{k}_{t-1}-\mathbf{w}^k_{t})^T\bm{\delta}_{t}^k$, we define an auxiliary vector as
\begin{align}
\hat{\mathbf{w}}^k_t=\Pi_{\mathcal{W}}(\mathbf{w}^k_{t-1}-\eta\nabla f(\mathbf{w}^k_{t-1})).\nonumber
\end{align}
Thus, we have
\begin{align}
-\left(\bar{\mathbf{w}}^{k}_{t-1}-\mathbf{w}^k_{t}\right)^T\bm{\delta}_{t}^k&=(\mathbf{w}^k_{t}-\hat{\mathbf{w}}^k_t+\hat{\mathbf{w}}^k_t-\bar{\mathbf{w}}^{k}_{t-1})^T\bm{\delta}_{t}^k\nonumber\\
&\leq\|\mathbf{w}^k_{t}-\hat{\mathbf{w}}^k_t\|\|\bm{\delta}_{t}^k\|+(\hat{\mathbf{w}}^k_t-\bar{\mathbf{w}}^{k}_{t-1})^T\bm{\delta}_{t}^k\nonumber\\
&\leq\|\mathbf{w}^k_{t-1}-\eta\mathbf{v}_{t}^k-(\mathbf{w}^k_{t-1}-\eta\nabla f(\mathbf{w}^k_{t-1}))\|\|\bm{\delta}_{t}^k\|+(\hat{\mathbf{w}}^k_t-\bar{\mathbf{w}}^{k}_{t-1})^T\bm{\delta}_{t}^k\nonumber\\
&=\eta\|\bm{\delta}_{t}^k\|^2+(\hat{\mathbf{w}}^k_t-\bar{\mathbf{w}}^{k}_{t-1})^T\bm{\delta}_{t}^k,\nonumber
\end{align}
where the second inequality is due to the non-expansive property of projection (Proposition B.11(c) in \cite{bertsekas1999nonlinear}).
The above inequality and \EqRef{eq:iteratediff} imply that
\begin{align}
\left\|\mathbf{w}^k_t-\bar{\mathbf{w}}^{k}_{t}\right\|^2\leq\left\|\mathbf{w}^k_{t-1}-\bar{\mathbf{w}}^{k}_{t-1}\right\|^2-2\eta\left(f(\mathbf{w}^k_{t})-f^\star\right)
+2\eta^2\|\bm{\delta}_{t}^k\|^2+2\eta(\hat{\mathbf{w}}^k_t-\bar{\mathbf{w}}^{k}_{t-1})^T\bm{\delta}_{t}^k.\label{eq:deltaineq}
\end{align}
Considering \LemmaRef{lemma:stocgradvar} with $\bm{\delta}_{t}^k=\nabla f(\mathbf{w}^k_{t-1})-\mathbf{v}^k_{t}$ and noticing that $\hat{\mathbf{w}}^k_t-\bar{\mathbf{w}}^{k}_{t-1}$
is independent of the random variable $i^k_t$ and $\mathcal{F}^{k}_{t}=\mathcal{F}^{k}_{t-1}\cup\{i^k_t\}$, we have
$\Expect{\mathcal{F}^k_t}{\|\bm{\delta}_{t}^k\|^2\mid\mathcal{F}^{k}_{t-1}}\leq4L_P\left(f(\mathbf{w}_{t-1}^k)-f^{\star}+f(\tilde{\mathbf{w}}^{k-1})-f^{\star}\right)$ and  $\Expect{\mathcal{F}^k_t}{(\hat{\mathbf{w}}^k_t-\bar{\mathbf{w}}^{k}_{t-1})^T\bm{\delta}_{t}^k\mid\mathcal{F}^{k}_{t-1}}=(\hat{\mathbf{w}}^k_t-\bar{\mathbf{w}}^{k}_{t-1})^T\Expect{\mathcal{F}^k_t}{\bm{\delta}_{t}^k\mid\mathcal{F}^{k}_{t-1}}=\mathbf{0}$.
Taking expectation with respect to $\mathcal{F}^{k}_{t}$ conditioned on $\mathcal{F}^{k}_{t-1}$ on both sides of \EqRef{eq:deltaineq}, we have
\begin{align}
\Expect{\mathcal{F}^k_t}{\left\|\mathbf{w}^k_t-\bar{\mathbf{w}}^{k}_{t}\right\|^2\mid\mathcal{F}^{k}_{t-1}}\leq &\left\|\mathbf{w}^k_{t-1}-\bar{\mathbf{w}}^{k}_{t-1}\right\|^2-2\eta\Expect{\mathcal{F}^k_t}{f(\mathbf{w}^k_{t})-f^\star\mid\mathcal{F}^{k}_{t-1}}\nonumber\\
&+2\eta^2\Expect{\mathcal{F}^k_t}{\|\bm{\delta}_{t}^k\|^2\mid\mathcal{F}^{k}_{t-1}}+2\eta(\hat{\mathbf{w}}^k_t-\bar{\mathbf{w}}^{k}_{t-1})^T\Expect{\mathcal{F}^k_t}{\bm{\delta}_{t}^k\mid\mathcal{F}^{k}_{t-1}}\nonumber\\
\leq &\left\|\mathbf{w}^k_{t-1}-\bar{\mathbf{w}}^{k}_{t-1}\right\|^2-2\eta\Expect{\mathcal{F}^k_t}{f(\mathbf{w}^k_{t})-f^\star\mid\mathcal{F}^{k}_{t-1}}\nonumber\\
&+8L_P\eta^2\left(f(\mathbf{w}_{t-1}^k)-f^{\star}+f(\tilde{\mathbf{w}}^{k-1})-f^{\star}\right).\nonumber
\end{align}
Taking expectation with respect to $\mathcal{F}^{k}_{t-1}$ on both sides of the above inequality and considering the fact that $\mathbb{E}_{\mathcal{F}^k_{t-1}}\left[\Expect{\mathcal{F}^k_{t}}{\left\|\mathbf{w}^k_t-\bar{\mathbf{w}}^{k}_{t}\right\|^2\mid\mathcal{F}^{k}_{t-1}}\right]=\Expect{\mathcal{F}^k_{t}}{\left\|\mathbf{w}^k_t-\bar{\mathbf{w}}^{k}_{t}\right\|^2}$,
we have
\begin{align}
\Expect{\mathcal{F}^k_t}{\left\|\mathbf{w}^k_t-\bar{\mathbf{w}}^{k}_{t}\right\|^2}\leq&\Expect{\mathcal{F}^k_{t-1}}{\left\|\mathbf{w}^k_{t-1}-\bar{\mathbf{w}}^{k}_{t-1}\right\|^2}-2\eta\Expect{\mathcal{F}^k_{t}}{f(\mathbf{w}^k_{t})-f^\star}\nonumber\\
&+8L_P\eta^2\Expect{\mathcal{F}^k_{t-1}}{f(\mathbf{w}_{t-1}^k)-f^{\star}+f(\tilde{\mathbf{w}}^{k-1})-f^{\star}}.\nonumber
\end{align}
Summing the above inequality over $t=1,2,\cdots,m$ by noticing that $\mathcal{F}^{k}_{0}=\mathcal{F}^{k-1}_{m}$, we have
\begin{align}
&\Expect{\mathcal{F}^k_m}{\left\|\mathbf{w}^k_m-\bar{\mathbf{w}}^{k}_{m}\right\|^2}+2\eta\sum_{t=1}^{m}\Expect{\mathcal{F}^k_t}{f(\mathbf{w}^k_{t})-f^\star}\nonumber\\
\leq &\Expect{\mathcal{F}^{k-1}_m}{\left\|\mathbf{w}^k_0-\bar{\mathbf{w}}^{k}_{0}\right\|^2}+8L_P\eta^2\sum_{t=1}^m\Expect{\mathcal{F}^k_{t-1}}{f(\mathbf{w}_{t-1}^k)-f^{\star}}+8L_P\eta^2m\Expect{\mathcal{F}^k_{t-1}}{f(\tilde{\mathbf{w}}^{k-1})-f^{\star})},\nonumber
\end{align}
Thus, we have
\begin{align}
&\Expect{\mathcal{F}^k_m}{\left\|\mathbf{w}^k_m-\bar{\mathbf{w}}^{k}_{m}\right\|^2}+2\eta\Expect{\mathcal{F}^k_m}{f(\mathbf{w}_{m}^k)-f^{\star}}+2\eta(1-4L_P\eta)\sum_{t=1}^{m-1}\Expect{\mathcal{F}^k_t}{f(\mathbf{w}_{t}^k)-f^{\star}}\nonumber\\
\leq &\Expect{\mathcal{F}^{k-1}_m}{\left\|\mathbf{w}^k_0-\bar{\mathbf{w}}^{k}_{0}\right\|^2}+8L_P\eta^2\Expect{\mathcal{F}^k_{t-1}}{f(\mathbf{w}_{0}^k)-f^{\star}+m(f(\tilde{\mathbf{w}}^{k-1})-f^{\star})},\nonumber
\end{align}
which together with $\Expect{\mathcal{F}^k_m}{\left\|\mathbf{w}^k_m-\bar{\mathbf{w}}^{k}_{m}\right\|^2}\geq0$, $2\eta\Expect{\mathcal{F}^k_m}{f(\mathbf{w}_{m}^k)-f^{\star}}\geq0$, $2\eta>2\eta(1-4L_P\eta)>0$ and $\mathbf{w}^k_0=\tilde{\mathbf{w}}^{k-1}$ implies that
\begin{align}
&2\eta(1-4L_P\eta)\sum_{t=1}^{m}\Expect{\mathcal{F}^k_m}{f(\mathbf{w}_{t}^k)-f^{\star}}\nonumber\\
\leq & \Expect{\mathcal{F}^{k-1}_m}{\left\|\mathbf{w}^k_0-\bar{\mathbf{w}}^{k}_{0}\right\|^2}
+8L_P\eta^2(m+1)\Expect{\mathcal{F}^{k-1}_m}{f(\tilde{\mathbf{w}}^{k-1})-f^{\star}},\label{eq:objdiffineq}
\end{align}
where we use the fact that $\Expect{\mathcal{F}^{k}_{t-1}}{f(\mathbf{w}^{k}_0)-f^{\star}}=\Expect{\mathcal{F}^{k}_{t-1}}{f(\tilde{\mathbf{w}}^{k-1})-f^{\star}}=\Expect{\mathcal{F}^{k-1}_m}{f(\tilde{\mathbf{w}}^{k-1})-f^{\star}}$.
By the convexity of $f(\cdot)$, we have
\begin{align}
f(\tilde{\mathbf{w}}^k)=f\left(\frac{1}{m}\sum_{t=1}^{m}\mathbf{w}_{t}^k\right)\leq\frac{1}{m}\sum_{t=1}^{m}f(\mathbf{w}_{t}^k).\nonumber
\end{align}
Thus, we have
\begin{align}
m\left(f(\tilde{\mathbf{w}}^k)-f^{\star}\right)\leq\sum_{t=1}^{m}\left(f(\mathbf{w}_{t}^k)-f^{\star}\right),\label{eq:avgconvexityineq}
\end{align}
Considering \LemmaRef{lemma:strongineq} with bounded $\{\tilde{\mathbf{w}}^{k-1}\}$, $\tilde{\mathbf{w}}^{k-1}=\mathbf{w}^k_{0}\in\mathcal{W}$ and $\bar{\mathbf{w}}^k_{0}=\Pi_{\mathcal{W}^\star_c}(\mathbf{w}^k_{0})$, we have
\begin{align}
f(\tilde{\mathbf{w}}^{k-1})-f^{\star}=f(\mathbf{w}^k_{0})-f^{\star}\geq\frac{\beta}{2}\left\|\mathbf{w}^k_{0}-\bar{\mathbf{w}}^k_{0}\right\|^2,\nonumber
\end{align}
which together with Eqs. (\ref{eq:objdiffineq}), (\ref{eq:avgconvexityineq}) implies that
\begin{align}
&2\eta(1-4L_P\eta)m\Expect{\mathcal{F}^k_m}{f(\tilde{\mathbf{w}}^k)-f^{\star}}\leq\Expect{\mathcal{F}^{k-1}_m}{\left\|\mathbf{w}^{k}_{0}-\bar{\mathbf{w}}^{k}_{0}\right\|^2}\nonumber\\
&+8L_P\eta^2(m+1)\Expect{\mathcal{F}^{k-1}_m}{f(\tilde{\mathbf{w}}^{k-1})-f^{\star}}
\leq\left(8L_P\eta^2(m+1)+\frac{2}{\beta}\right)\Expect{\mathcal{F}^{k-1}_m}{f(\tilde{\mathbf{w}}^{k-1})-f^{\star}}.\nonumber
\end{align}
Thus, we have
\begin{align}
\Expect{\mathcal{F}^k_m}{f(\tilde{\mathbf{w}}^k)-f^{\star}}\leq\left(\frac{4L_P\eta(m+1)}{(1-4L_P\eta)m}+\frac{1}{\beta\eta(1-4L_P\eta)m}\right)\Expect{\mathcal{F}^{k-1}_m}{f(\tilde{\mathbf{w}}^{k-1})-f^{\star}}.\nonumber
\end{align}
Using the above recursive relation and considering the definition of $\rho$ in \EqRef{eq:linearrate}, we complete the proof of the theorem.
\qed
\end{proof}
\begin{remark}\label{remark:strongcondition}
If $f$ is strongly convex with parameter $\tilde{\mu}$, then the inequality in \LemmaRef{lemma:strongineq} holds with $\beta=\tilde{\mu}$. Therefore, we can easily obtain from the proof of \ThemRef{theorem:linearconvergence} that
\begin{align}
\Expect{\mathcal{F}^k_m}{f(\tilde{\mathbf{w}}^k)-f^{\star}}\leq\left(\frac{4L_P\eta(m+1)}{(1-4L_P\eta)m}+\frac{1}{\tilde{\mu}\eta(1-4L_P\eta)m}\right)^k(f(\tilde{\mathbf{w}}^{0})-f^{\star}),\nonumber
\end{align}
which has the same convergence rate as \cite{xiao2014proximal}.
\end{remark}

\begin{proof}[Proof of \ThemRef{theorem:reglinearconvergence}]
We know that the sequence $\{\mathbf{w}^k_t\}$ generated by the proximal step in \EqRef{eq:proxstep} is bounded, which together with \LemmaRef{lemma:regstrongineq} implies that
\begin{align}
F(\mathbf{w}^k_t)-F^\star\geq\frac{\beta}{2}\|\mathbf{w}^k_t-\Pi_{\mathcal{W}^\star_r}(\mathbf{w}^k_t)\|^2,~\forall k,t\geq0.\nonumber
\end{align}
We also note that Lemmas \ref{lemma:stocgradvar}, \ref{lemma:interbound}
are established for constrained optimization problems which are adapted from Corollary 3.5 and Lemma 3.7
for regularized optimization problems in \cite{xiao2014proximal}. Thus, similar inequalities in Lemmas \ref{lemma:stocgradvar}, \ref{lemma:interbound}
also hold for the regularized problem in \EqRef{eq:regularizedprob}. Therefore, each step in the proof of \ThemRef{theorem:linearconvergence} is true by replacing $f(\cdot)$ in \EqRef{eq:constrainedprob} and the projection step with $F(\cdot)$ in \EqRef{eq:regularizedprob} and the proximal step, respectively.
This completes the proof of the theorem.
\qed
\end{proof}

\section{More Experimental Results}\label{appendix:moreresults}
We conduct sensitivity studies for VRPSG on the sampling distribution parameter $\mathbf{p}=[p_1,\cdots,p_n]^T$, the inner iterative number $m$ and the step size $\eta$ by varying one parameter and keeping the other two parameters fixed.
We report the objective function value $f(\tilde{\mathbf{w}}^k)$ vs. the number of gradient evaluations ($\sharp$grad/n) plots in \FigRef{fig:pselection}, \FigRef{fig:mselection} and \FigRef{fig:etaselection}. From these results, we have the following observations: (a) The VRPSG algorithm with
non-uniform sampling (i.e., $p_i=L_i/\sum_{i=1}^nL_i$) is much more efficient than that with uniform sampling (i.e., $p_i=1/n$), which is consistent with the analysis in the remarks of \ThemRef{theorem:linearconvergence}.
(b) In general, the VRPSG algorithm by setting $m=0.5n,n$ has the most stable performance, which indicates that a small or large $m$ will degrade the performance of the VRPSG algorithm.
(c) The optimal step sizes of the VRPSG algorithm on different data sets are slightly different. Moreover, the VRPSG algorithm with step sizes $\eta=1/L_P$ and $\eta=5/L_P$ converges quickly, which demonstrates that
the VRPSG algorithm still performs well even if the step size is much larger than that required in the theoretical analysis ($\eta<0.25/L_P$
is required in \ThemRef{theorem:linearconvergence}). This shows the robustness of the VRPSG algorithm.

\begin{figure*}[!ht]\vspace{-0.0cm}
\begin{minipage}[c]{1.0\linewidth}
\centering
\includegraphics[width=.32\linewidth]{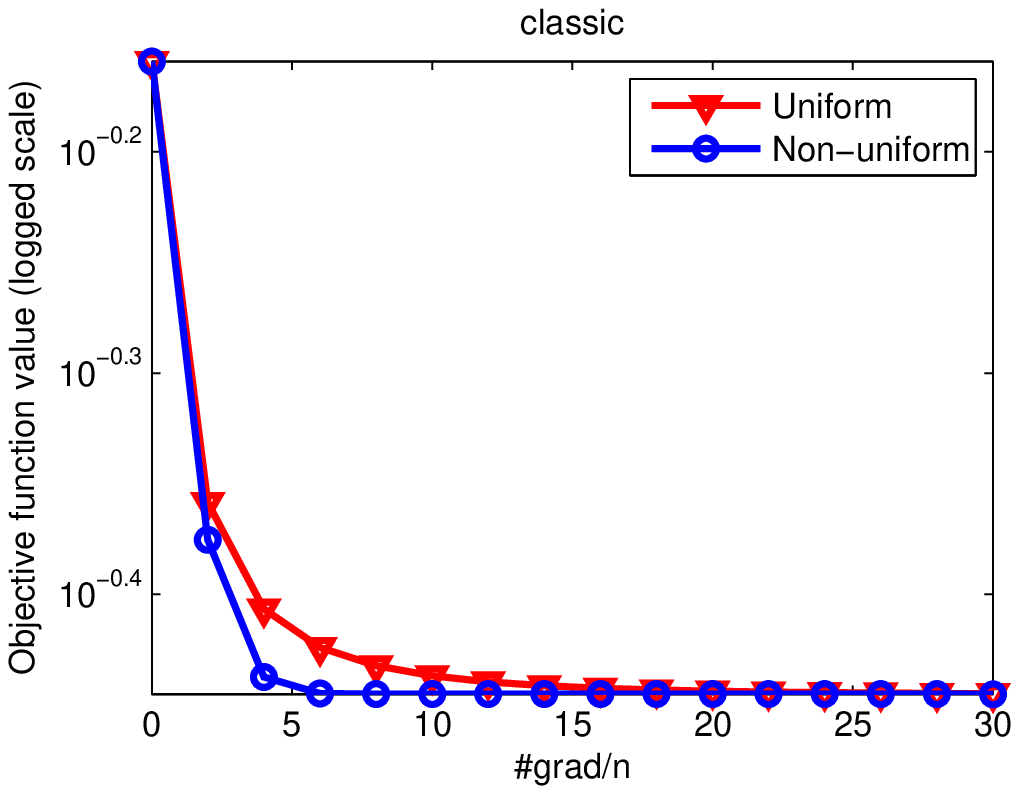}
\includegraphics[width=.32\linewidth]{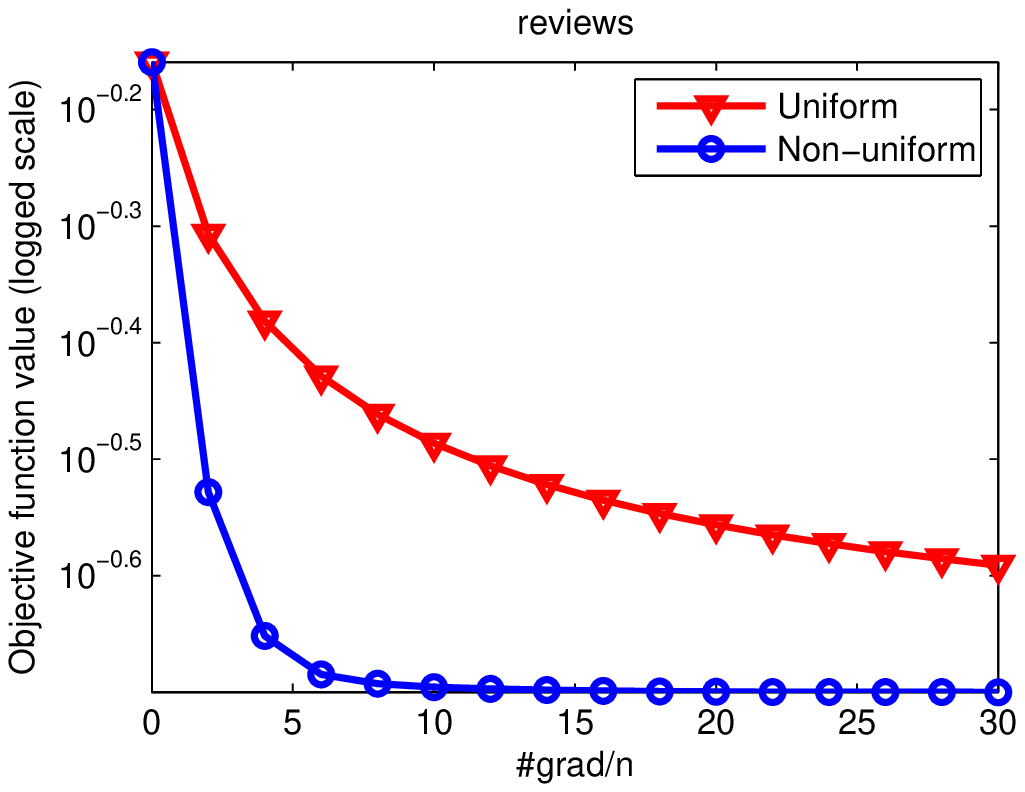}
\includegraphics[width=.32\linewidth]{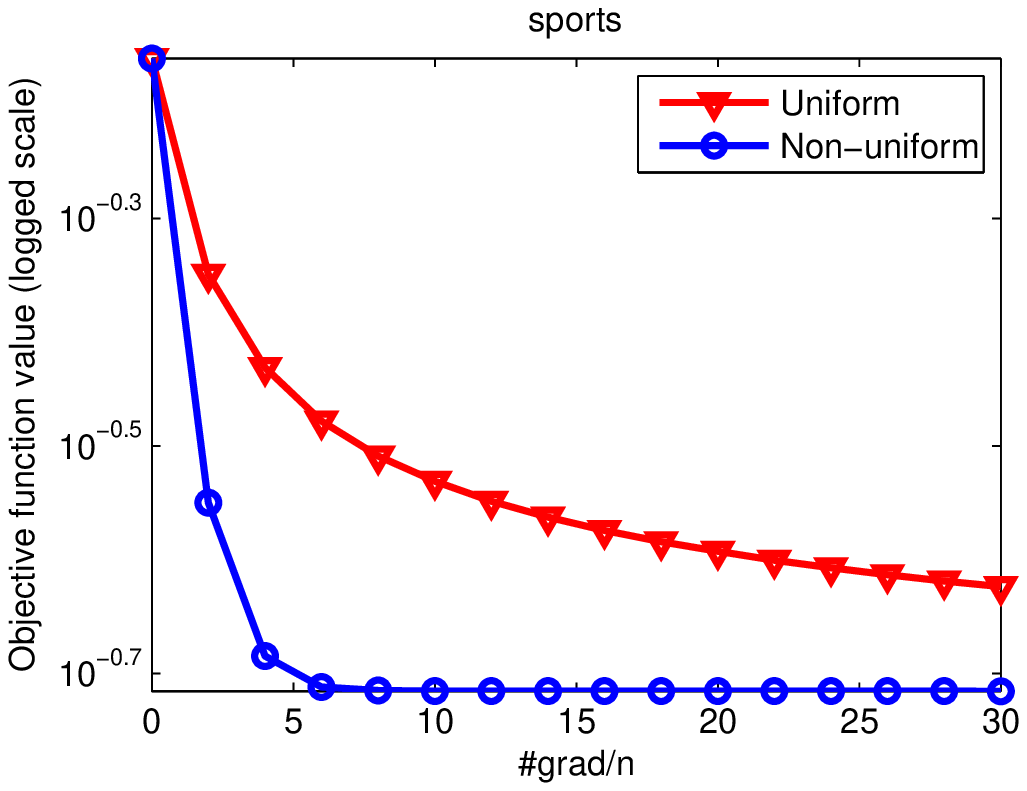}
\end{minipage}
\vskip 0.0cm
\vspace{-0.1cm}\caption{Sensitivity study of VRPSG on the parameter $\mathbf{p}=[p_1,\cdots,p_n]^T$: the objective function value $f(\tilde{\mathbf{w}}^k)$ vs. the number of gradient evaluations ($\sharp$grad/n) plots (averaged on 10 runs).
``Uniform'' and ``Non-uniform'' indicate that $p_i=1/n$ and $p_i=L_i/\sum_{i=1}^nL_i$, respectively. Other parameters are set as $\tau=10$, $m=n$, $\eta=1/L_P$.}
\label{fig:pselection}\vspace{-0.2cm}
\end{figure*}

\begin{figure*}[!ht]\vspace{-0.0cm}
\begin{minipage}[c]{1.0\linewidth}
\centering
\includegraphics[width=.32\linewidth]{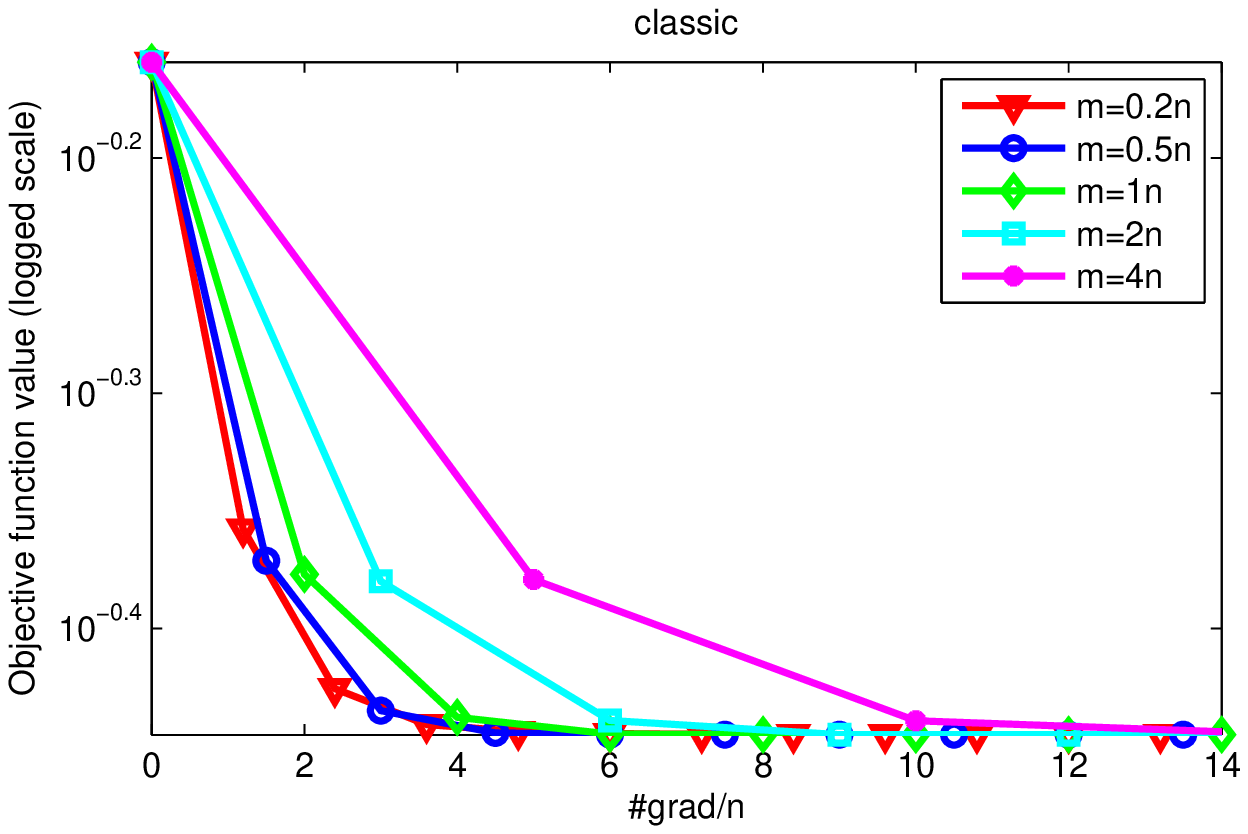}
\includegraphics[width=.32\linewidth]{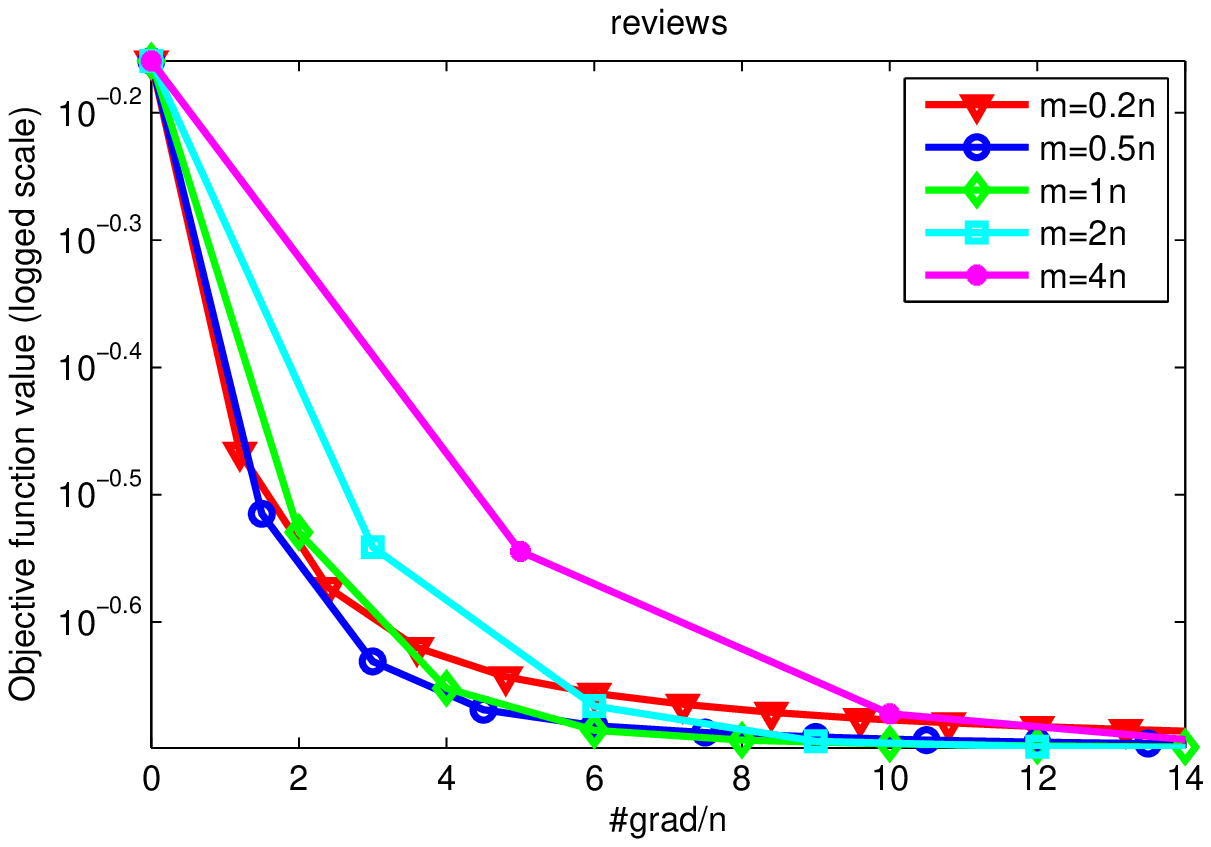}
\includegraphics[width=.32\linewidth]{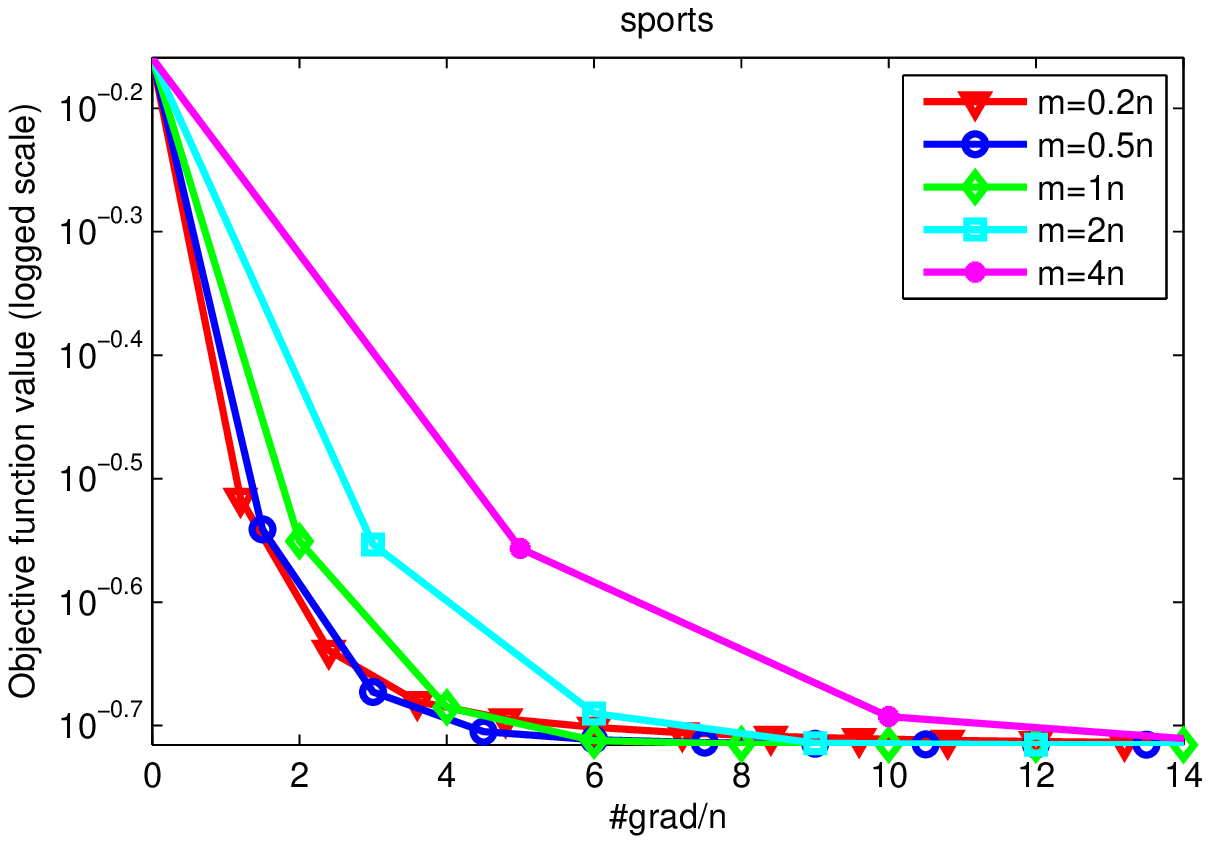}
\end{minipage}
\vskip 0.0cm
\vspace{-0.1cm}\caption{Sensitivity study of VRPSG on the parameter $m$: the objective function value $f(\tilde{\mathbf{w}}^k)$ vs. the number of gradient evaluations ($\sharp$grad/n) plots (averaged on 10 runs).
Other parameters are set as $\tau=10$, $p_i=L_i/\sum_{i=1}^nL_i$, $\eta=1/L_P$.}
\label{fig:mselection}\vspace{-0.2cm}
\end{figure*}

\begin{figure*}[!ht]\vspace{-0.0cm}
\begin{minipage}[c]{1.0\linewidth}
\centering
\includegraphics[width=.32\linewidth]{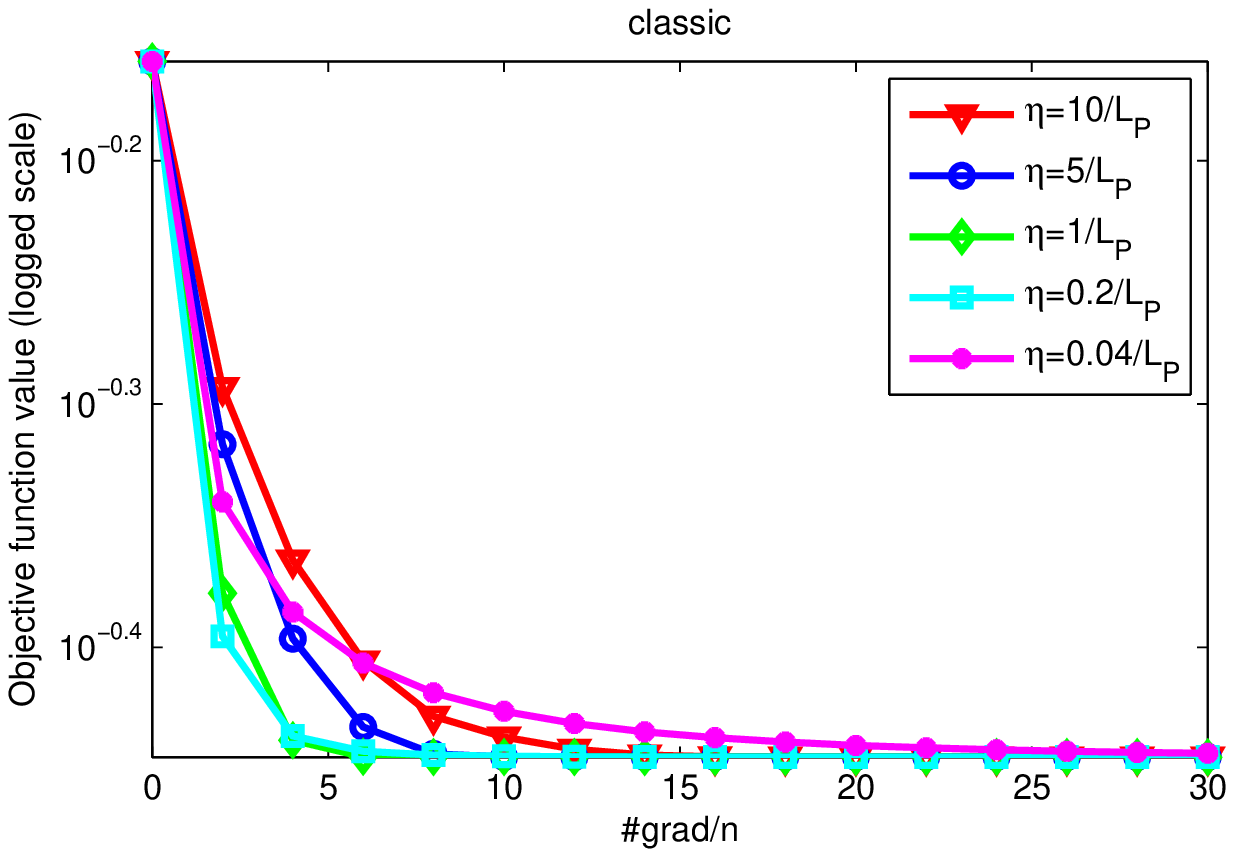}
\includegraphics[width=.32\linewidth]{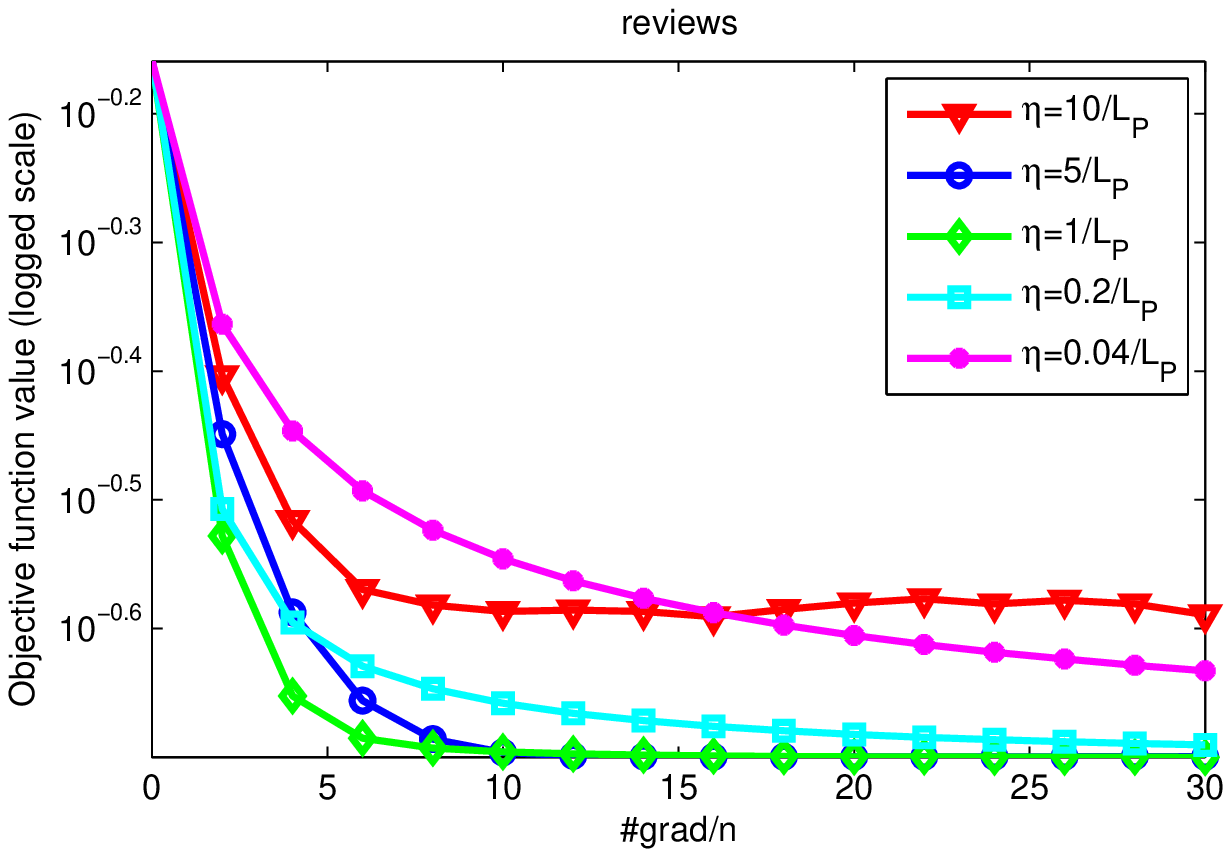}
\includegraphics[width=.32\linewidth]{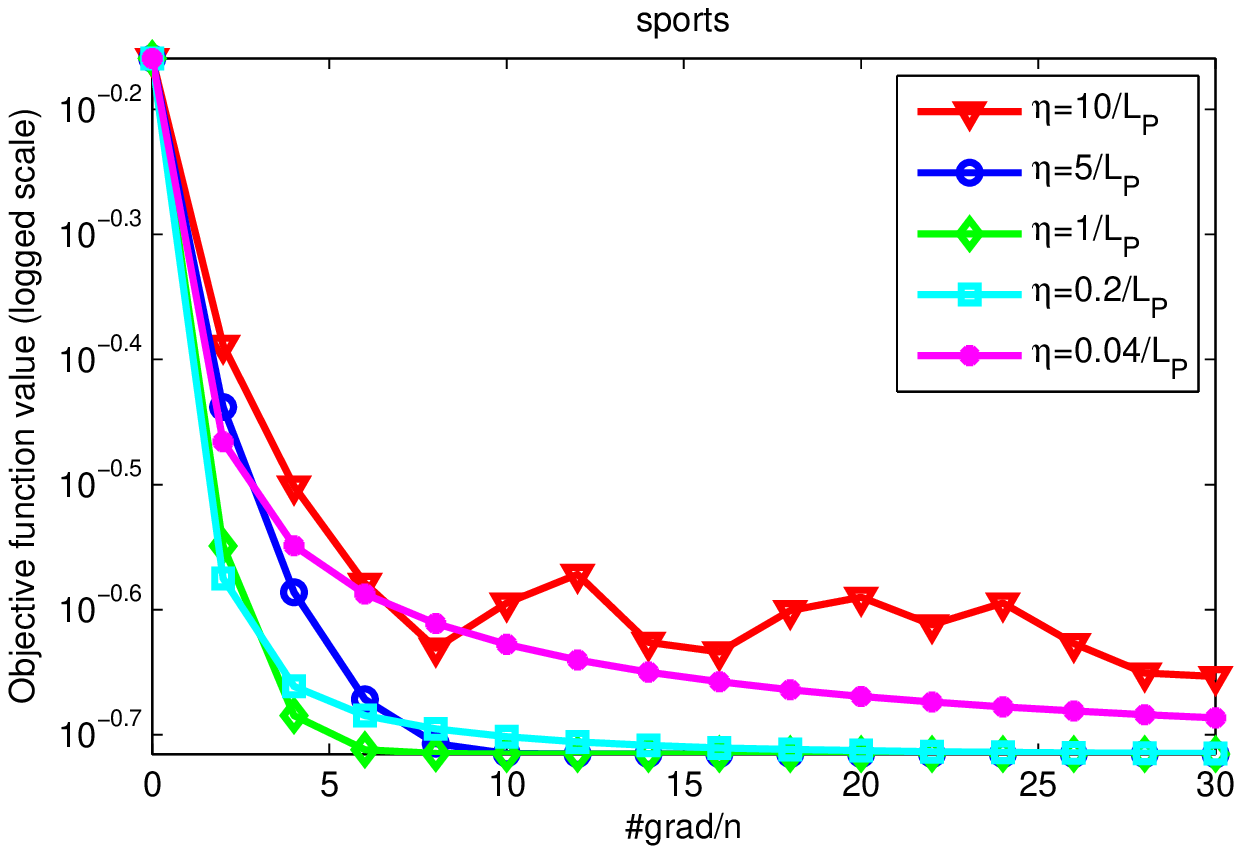}
\end{minipage}
\vskip 0.0cm
\vspace{-0.1cm}\caption{Sensitivity study of VRPSG on the parameter $\eta$: the objective function value $f(\tilde{\mathbf{w}}^k)$ vs. the number of gradient evaluations ($\sharp$grad/n) plots (averaged on 10 runs).
Other parameters are set as $\tau=10$, $m=n$, $p_i=L_i/\sum_{i=1}^nL_i$.}
\label{fig:etaselection}\vspace{-0.2cm}
\end{figure*}

\end{document}